\documentclass{siamltex}
\usepackage{amsmath,amssymb}
\usepackage{graphicx}
\newcommand{\hu}{\hat{u}}

\newcommand{\vepsi}{{\varepsilon}}
\newcommand{\eps}{\varepsilon}
\newcommand{\Ome}{{\Omega}}

\newcommand{\p}{{\partial}}
\newcommand{\Del}{{\Delta}}

\newcommand{\nab}{\nabla}

\newcommand{\ue}{u^\eps}

\newtheorem{remark}{Remark}[section]

\newtheorem{thm}{Theorem}[section]
\newtheorem{lem}[thm]{Lemma}
\newtheorem{cor}[thm]{Corollary}

\title{Convergence of a fourth order singular perturbation of 
the $n$-dimensional radially symmetric Monge-Amp\`ere equation}

\author{Xiaobing Feng\thanks{Department of Mathematics, The University of
Tennessee, Knoxville, TN 37996 (xfeng@math.utk.edu). The work of this author was
partially supported by the NSF grant DMS-071083.}
\and{Michael Neilan}
\thanks{Department of Mathematics, University of Pittsburgh, 
Pittsburgh, PA 15260 (neilan@pitt.edu).  The work of this author was 
partially supported by the NSF grant DMS-1238711.}
}

\begin{document}

\maketitle

\begin{abstract}
This paper concerns with the convergence analysis of a fourth order
singular perturbation of the Dirichlet Monge-Amp\`ere problem in the 
$n$-dimensional radial symmetric case. A detailed study of the fourth 
order problem is presented.  In particular, various {\em a priori} estimates with 
explicit dependence on the perturbation parameter $\vepsi$ are derived,
and a crucial convexity property is also proved for the solution of   
the fourth order problem. Using these estimates and the convexity 
property, we prove that the solution of the perturbed problem 
converges uniformly and compactly to the unique convex viscosity  
solution of the Dirichlet Monge-Amp\`ere problem. Rates of convergence
in the $H^k$-norm for $k=0,1,2$ are established, and illustrating 
numerical experiment results are also presented in the paper. 
\end{abstract}

\begin{keywords}
Monge-Amp\`ere equation, singular perturbation, vanishing moment method,
convergence and rate of convergence 
\end{keywords}

\begin{AMS}
35J96  
35J30  
65N30  
\end{AMS}

\section{Introduction}\label{section-1}
In this paper we consider the following fourth order singularly perturbed 
Monge-Amp\`ere problem:
\begin{subequations}\label{fourthorder}
\begin{alignat}{2}\label{fourthorder1}
-\eps\Delta^2 u^\eps +\det(D^2 u^\eps)& =f \qquad &&\text{in}\ \Omega,\\ %
u^\eps& = g\qquad &&\text{on }\p \Ome, \label{fourthorder2}\\
\Delta u^\eps &=\eps \qquad &&\text{on }\p\Ome, \label{fourthorder3}
\end{alignat}
\end{subequations}
where $\vepsi>0$, $\Omega\subset \mathbf{R}^n$ is a convex domain 
and $f>0$ in $\Ome$.
The above problem is a special case of the following fourth order singular 
perturbation problem:
\begin{subequations}\label{moment}
\begin{alignat}{2}\label{moment1}
\eps\Delta^2u^\eps+F(D^2u,\nabla u^\eps,u^\eps,x)&=0\qquad &&\text{in }\Ome,\\
\label{moment2}u^\eps&=g\qquad &&\text{on }\p\Ome,\\
\label{moment3}\Delta u^\eps&=\eps
\quad&&\text{on}\ \p\Ome
\end{alignat}
\end{subequations}
with $F(D^2 u^\eps,\nabla u^\eps,u^\eps,x)=f(x)-\det(D^2 u^\eps)$. 
Here, $D^2 u^\eps(x)$ and $\nab u^\eps(x)$ denote respectively the 
Hessian matrix and the gradient of $u^\eps$ at $x\in \Ome$.

Problem \eqref{fourthorder}, respectively \eqref{moment}, was proposed by the 
authors in \cite{Feng_Neilan09a} as a stepstone to develop efficient numerical 
methods, particularly Galerkin-type methods,  
for approximating the viscosity solution 
of the Dirichlet Monge-Amp\`ere problem (cf. \cite{Feng_Neilan09b,Feng_Neilan11}):
\begin{subequations}\label{MAeqn}
\begin{alignat}{2}
\label{MAeqn1}\text{det}(D^2 u)&=f\qquad &&\text{in}\ \Ome,\\
\label{MAeqn2}u&=g\qquad && \text{on}\ \p\Ome,
\end{alignat}
\end{subequations}
respectively, the general fully nonlinear second order PDE problem:
\begin{subequations}\label{generalPDE} 
\begin{alignat}{2}
F(D^2 u, \nab u, u, x)&=0\qquad &&\text{in }\Ome,\\
u & = g\qquad &&\text{on }\p \Ome.
\end{alignat}
\end{subequations}
We remark that since the boundary condition \eqref{fourthorder3} 
(resp., \eqref{moment3}) is artificially imposed to make the perturbed problem 
well-posed; other type boundary conditions can be used in the place of 
\eqref{fourthorder3} (resp., \eqref{moment3}) (cf. \cite{Feng_Neilan09a}).

Fully nonlinear second order PDEs, which have experienced extensive analytical 
developments in the past thirty years 
(cf. \cite{Gilbarg_Trudinger01,Caffarelli_Cabre95}), 
arise from many scientific and engineering fields 
including differential geometry, optimal control, mass transportation, 
geostrophic fluid, meteorology, and general relativity 
(cf. \cite{Caffarelli_Milman99,Gilbarg_Trudinger01,Fleming_Soner06,Benamou_Brenier00,Trudinger_Wang08b,Brenier10,Feng_Glowinski_Neilan10} 
and the references therein). Fully nonlinear PDEs play a critical role 
in the solutions of these application problems because they appear one 
way or another in the governing equations. 
By the middle eighties the classical solution theory for fully 
nonlinear second order PDEs was  well established 
(cf. \cite[Chapter 17]{Gilbarg_Trudinger01}).  The classical framework was 
soon followed by a comprehensive  viscosity (weak) 
solution theory (cf. \cite{Crandall_Ishii_Lions92,Caffarelli_Milman99} 
and the references therein) after the introduction of 
viscosity solutions by Crandall and Lions in \cite{Crandall_Lions83} for 
fully nonlinear first order Hamilton-Jacobi equations.

In contrast with the success of PDE analysis, numerical solutions for general fully 
nonlinear second order PDEs was a relatively untouched area until very recently 
(cf. \cite{Feng_Glowinski_Neilan10} and the references therein). 
The lack of progress is mainly due to the following two facts: (i) the
notion of viscosity solutions is nonvariational; (ii) the conditional
uniqueness (i.e., uniqueness only holds in a restrictive function class)
of viscosity solutions is difficult to handle at the discrete level.
The first difficulty prevents a direct construction of Galerkin-type methods
and forces one to use indirect approaches as done in 
\cite{Glowinski09,Feng_Neilan09a,Feng_Neilan11} 
for approximating viscosity solutions; {\em this is main focus of this paper}. 
The second difficulty prevents 
any straightforward construction of finite difference methods because such 
a method does not have a mechanism to enforce the conditional uniqueness 
and often fails to capture the sought-after viscosity solution 
(cf. \cite{Feng_Glowinski_Neilan10} and the reference therein).
To overcome the above difficulties, inspired by {\em the vanishing viscosity 
method} for the Hamilton-Jacobi equations, in \cite{Feng_Neilan09a}
we proposed to approximate the fully nonlinear problem \eqref{MAeqn}
(resp. \eqref{generalPDE}) by the fourth order quasilinear problem
\eqref{fourthorder} (resp. \eqref{moment}).  We called this approximation 
procedure {\em the vanishing moment method}. If
the vanishing moment method converges, then the solution of
the former can be computed via the solution of the latter.
This can be achieved by using many existing numerical methodologies 
including all Galerkin-type methods. So far, all of the numerical 
experiments and the analysis of 
\cite{Feng_Neilan09a,Feng_Neilan09b,Feng_Neilan11,Neilan09,Neilan_thesis} 
show that the vanishing moment 
method works effectively.

The primary goal of this paper is to provide a rigorous proof 
of the convergence of the vanishing moment method applied to the 
Monge-Amp\'ere equation in the radially symmetric case in dimension $n\ge 2$. 
Specifically, we shall give a detailed analysis of problem 
\eqref{fourthorder} in the radially symmetric case and prove the 
convergence of the solution $u^\vepsi$ of problem \eqref{fourthorder}
to the unique convex viscosity solution $u$ of problem \eqref{MAeqn}. 
We also derive the rates of convergence 
for $u-u^\vepsi$ in various of norms. These results then put  
the vanishing moment method on a solid footing and provides a partial
theoretical foundation for the numerical work in 
\cite{Feng_Neilan09a,Feng_Neilan09b,Feng_Neilan11,Neilan09}.

The remainder of the paper is organized as follows. 
In Section \ref{section-2}, we give the radially symmetric descriptions
of problems \eqref{MAeqn} and \eqref{fourthorder}. Sections 
\ref{section-3} and \ref{section-4}, which are the bulk of this paper,
is devoted to studying the existence, uniqueness, regularity and
convexity of the solution of the singularly perturbed 
problem \eqref{fourthorder}. With these results in hand, 
we prove the uniform convergence of the solution $u^\vepsi$ of 
problem \eqref{fourthorder} to the unique convex viscosity solution 
$u$ of problem \eqref{MAeqn} in Section \ref{section-5}.
In addition, we derive rates of convergence for $u-u^\vepsi$ 
in the $H^k$-norm for $k=0,1,2$ in Section \ref{section-6}. 
We present a few illustrating
numerical computational results in Section \ref{section-7} and 
end the paper with some concluding remarks in Section \ref{section-8}.

\section{Descriptions of radially symmetric problems}\label{section-2}
Standard space notation is adopted in this paper, we refer the reader
to \cite{Brenner_Scott08,Gilbarg_Trudinger01} for their exact
definitions. Unless stated otherwise,
$\Ome=B_R(0)\subset \mathbf{R}^n$ ($n\geq 2$) stands for the ball centered at 
the origin with radius $R$. We do not assume $\Ome$ is the unit 
ball because many of our results will depend on the size of the radius $R$. 
The unlabeled letter $C$ is used to denote a generic positive constant independent of 
$\vepsi$ that may take on different values at different occurrences.

Suppose that $f=f(r), f\not\equiv 0$ and $g=g(r)$ in \eqref{MAeqn}, that is, 
$f$ and $g$ are radial.  Then the solution $u$ of \eqref{MAeqn}
is expected to be radial, namely, $u(x)$ is a function of 
$r:=|x|=\sqrt{\sum_{j=1}^n x_j^2}$. We set $\hu(r):=\hu(|x|)=u(x)$, and
for the reader's convenience, we now compute $\Del u, \Del^2 u$ and
$\mbox{det}(D^2u)$ in terms of $\hu$ (cf.\,\cite{Monn86,Rios_Sawyer08}). 
First, a direct calculation shows $\frac{\p r}{\p x_j} = {x_j}/{r}$ and 
$\frac{\p }{\p x_j}\Bigl(r^{-1}\Bigr) = -{x_j}/{r^3}$.
By the chain rule we have
\begin{align*}
\frac{\p u(x)}{\p x_j}&=\hu_r(r) \frac{\p r}{\p x_j}
=\hu_r(r) \frac{x_j}{r} ,\quad
\frac{\p^2 u(x)}{\p x_j\p x_i}
=\frac{1}{r}  \Bigl(\frac{1}{r} \hu_r(r)\Bigr)_r x_ix_j 
  + \frac{\hu_r(r)}{r}\delta_{ij}.
\end{align*}
Here, the subscripts stand for the derivatives
with respect to the subscript variables.

Noting that $D^2u(x)$ is a diagonal perturbation of
a scaled rank-one matrix $xx^T$, and since the eigenvalues of $xx^T$ 
are $0$ (with multiplicity $n-1$) and $|x|^2=r^2$ (with multiplicity 
$1$ and corresponding eigenvector $x$),
 the eigenvalues of $D^2u(x)$ are
$\lambda_1 :=\frac{\hu_r(r)}{r} + r\Bigl(\frac{1}{r} \hu_r(r)\Bigr)_r =\hu_{rr}(r)$
(with multiplicity $1$), and
$\lambda_2 :=\frac{\hu_r(r)}{r}$
(with multiplicity $n-1$).
Thus,
\begin{align*}
&\Del u(x) = 
\lambda_1+ (n-1)\lambda_2 =\hu_{rr}(r) +\frac{n-1}{r} \hu_r(r)
=\frac{1}{r^{n-1}} \bigl(r^{n-1} \hu_r\big)_r,\\ 
&\Del^2 u(x) =\Del (\Del u) 
=\Del\Bigl( \frac{1}{r^{n-1}} \bigl(r^{n-1} \hu_r \big)_r \Bigr)
=\frac{1}{r^{n-1}} \Bigl(r^{n-1} \Bigl(\frac{1}{r^{n-1}} \bigl(r^{n-1} \hu_r \big)_r
\Bigr)_r \Bigr)_r,\\ 
&\mbox{det}(D^2u(x)) = \lambda_1 (\lambda_2)^{n-1}
=\hu_{rr}(r) \Bigl[\frac{\hu_r(r)}{r} \Bigr]^{n-1}
=\frac{1}{n r^{n-1}}  \bigl((\hu_r \big)^n)_r.
\end{align*}

Abusing the notion by denoting $\hu(r)$ by $u(r)$, 
problem \eqref{MAeqn} becomes seeking 
a function $u=u(r)$ such that
\begin{subequations}
\label{MA_radial}
\begin{align}
\label{MA1_radial}
\frac{1}{n r^{n-1}}  \bigl((u_r \big)^n)_r &= f  \qquad\mbox{in } (0,R),\\
u(R) &=g(R), \label{MA2_radial} \\
u_r(0) &=0. \label{MA3_radial}
\end{align}
\end{subequations}
We remark that boundary condition \eqref{MA3_radial} is due to the symmetry
of $u=u(r)$. 

\begin{lem}\label{lem3.1}
Suppose that $r^{n-1}f\in L^1((0,R))$ and $f\geq 0$ a.e.\:on $(0,R)$. Then there
exists exactly one real solution if $n$ is odd and there are exactly
two real solutions if $n$ is even, to the boundary value problem 
\eqref{MA_radial}. Moreover, the solutions are
given by the formula
\begin{equation}\label{solution}
u(r)=\left\{ 
\begin{array}{ll}
g(R) \pm \int_r^R \bigl(nL_f(s)\bigr)^{\frac{1}{n}}\,\, ds
&\qquad\mbox{if $n$ is even},\\
\\
g(R) - \int_r^R \bigl(nL_f(s)\bigr)^{\frac{1}{n}}\,\, ds
&\qquad\mbox{if $n$ is odd}
\end{array}\right.
\end{equation}
for $r\in (0,R)$. Here, the function $L_f(s)$ is defined by
\begin{equation}\label{L_f}
L_f(s):=\int_0^s t^{n-1} f(t)\, dt.
\end{equation}
\end{lem}

Since the proof is elementary (cf.\:\cite{Monn86,Rios_Sawyer08}), 
we omit it.  Clearly, when $n$ is even, the first solution (with the ``$+$" sign)
is concave, and the second solution (with the ``$-$" sign) 
is convex because $u_{r}$ and $u_{rr}$ are simultaneously positive 
and negative respectively in the two cases. When $n$ is odd, the 
real solution is convex.

\begin{remark}
The above theorem shows that $u$ is $C^2$ at a point $r_0\in (0,R)$ as long as
$f$ is $C^0$ at $r_0$ and $L_f(r_0)\neq 0$. Also, $u$ is smooth in $(0,R)$ if $f$ is 
smooth in $(0,R)$.  We refer the reader to {\rm \cite{Monn86,Rios_Sawyer08}}
for the precise conditions on $f$ at $r=0$ to ensure the regularity
of $u$ at $r=0$, extensions to the complex Monge-Amp\`ere
equation, and generalized Monge-Amp\`ere equations in 
which $f=f(\nab u,u,x)$.
\end{remark}

Similarly, it is expected that $u^\vepsi=u^\vepsi(r)$ is also radial, and 
the vanishing moment approximation \eqref{moment}
then becomes
\begin{subequations}
\label{moment_radial}
\begin{align}\label{moment1_radial}
-\frac{\vepsi}{r^{n-1}} \Bigl(r^{n-1} \Bigl(\frac{1}{r^{n-1}} 
\bigl(r^{n-1} u^\vepsi_r \big)_r \Bigr)_r \Bigr)_r
+ \frac{1}{n r^{n-1}}  \bigl((u^\vepsi_r \big)^n)_r  &= f  \qquad\mbox{in } (0,R),\\
u^\vepsi(R) &=g(R), \label{moment2_radial} \\
u^\vepsi_r(0)=0,\quad |u^\vepsi_{rr}(0)|<\infty, 
\quad |u^\vepsi_{rrr}(r)| =o\bigl(\frac{1}{r^{n-1}}\bigr) &\quad\mbox{as }r\to 0^+,
\label{moment3_radial} \\
u^\vepsi_{rr}(R) + \frac{n-1}{R} u^\vepsi_r(R) &=\vepsi. \label{moment4_radial} 
\end{align}
\end{subequations}
In the next sections, we shall analyze problem 
\eqref{moment_radial}
which includes proving the existence, uniqueness and regularity of the solution.
After this is done, we then show that the solution
$u^\vepsi$ of \eqref{moment_radial}
converges to the unique convex solution of \eqref{MA_radial}.


Integrating over $(0,r)$ after multiplying \eqref{moment1_radial} by $r^{n-1}$,
using boundary condition \eqref{moment3_radial} and
$
\lim_{r\to 0^+} r^{n-1}\Bigl( u^\vepsi_{rrr}+ \frac{n-1}{r} u^\vepsi_{rr} 
-\frac{n-1}{r^2} u^\vepsi_r \Bigr)=0,
$
we get
\begin{equation}\label{moment1b_radial}
-\vepsi r^{n-1} \Bigl(\frac{1}{r^{n-1}} 
\bigl(r^{n-1} u^\vepsi_r \big)_r \Bigr)_r 
+\frac1{n} (u^\vepsi_r)^n = L_f 
\qquad \mbox{in } (0,R).
\end{equation}

Introduce the new function $w^\vepsi(r):=r^{n-1} u^\vepsi_r(r)$.  
A direct calculation shows that $w^\vepsi$ satisfies
\begin{subequations}
\label{moment_wradial}
\begin{equation}\label{moment1c_radial}
-\vepsi r^{n-1} \Bigl(\frac{1}{r^{n-1}} w^\vepsi_r\Bigr)_r
+\frac{1}{nr^{n(n-1)}} (w^\vepsi)^n = L_f 
\qquad \mbox{in } (0,R).
\end{equation}
Converting the boundary conditions 
\eqref{moment3_radial}--\eqref{moment4_radial} to $w^\vepsi$ we have
\begin{align}
w^\vepsi(0)=w^\vepsi_r(0) &=0, \label{moment3a_radial} \\
w^\vepsi_r(R) &=\vepsi R^{n-1}. \label{moment4a_radial} 
\end{align}
\end{subequations}
In addition, since 
$w^\vepsi_r=r^{n-1} u^\vepsi_{rr} +(n-1)r^{n-2} u^\vepsi_r,$ and
$w^\vepsi_{rr}=r^{n-1}u^\vepsi_{rrr} 
+ 2(n-1)r^{n-2}u^\vepsi_{rr} + (n-1)(n-2)r^{n-3}u^\vepsi_r,$
we have
\begin{align}\label{moment5a_radial}
\frac{\p^j w^\vepsi}{\p r^j}(0)=o\Bigl(\frac{1}{r^{n-1-j}}\Bigr)
\qquad\text{for } 0\leq j\leq \min\{2, n-1\}.
\end{align}

In summary, we have derived from \eqref{moment1_radial} a reduced equation 
\eqref{moment1c_radial}, which is second order and
easier to handle. After problem \eqref{moment_wradial}
is fully understood, we then come back to analyze problem 
\eqref{moment_radial}.

\section{Existence, uniqueness, and regularity of vanishing moment approximations} 
\label{section-3}
We now prove that problem \eqref{moment_wradial} possesses a unique nonnegative 
classical solution. First, we state and prove the following uniqueness result. 

\begin{thm}\label{uniqueness_thm}
Problem \eqref{moment_wradial} has at most one
nonnegative classical solution.
\end{thm}

\begin{proof}
Suppose that $w^\vepsi_1$ and $w^\vepsi_2$ are two nonnegative 
classical solutions to \eqref{moment1c_radial}--\eqref{moment4a_radial}. 
Let 
\[
\phi^\vepsi:= w^\vepsi_1-w^\vepsi_2\qquad\text{and}\qquad 
\overline{w}^\vepsi:=
\left\{
\begin{array}{ll}
\displaystyle\mathop{\sum_{\alpha+\beta = n-1}}_{\alpha,\beta\ge 0} (w_1^\eps)^\alpha (w_2^\eps)^\beta &\text{if }w^\eps_1= w^\eps_2\smallskip\\
\dfrac{(w^\vepsi_1)^n-(w^\vepsi_2)^n}{w^\vepsi_1-w^\vepsi_2} & \text{otherwise}.
\end{array}
\right.
\]
Subtracting the corresponding equations satisfied by 
$w^\vepsi_1$ and $w^\vepsi_2$ yields
\begin{subequations}
\label{ErrorEquationsABC}
\begin{align} \label{error1}
-\vepsi r^{n-1} \Bigl(\frac{1}{r^{n-1}} \phi^\vepsi_r\Bigr)_r
+ \frac{1}{nr^{n(n-1)}} \overline{w}^\vepsi \phi^\vepsi &= 0 
\qquad \mbox{in } (0,R), \\
\phi^\vepsi(0)=\phi^\vepsi_r(0) &=0,  \label{error2}\\
\phi^\vepsi_r(R) &=0. \label{error3}
\end{align}
\end{subequations}
Noting that $\overline{w}^\vepsi\geq 0$ in $[0,R]$, 
we have by the weak maximum 
principle \cite[Theorem 2, page 329]{Evans98}
$\max_{ [0,R]} |\phi^\vepsi (r)| =\max\{ |\phi^\vepsi(0)|, |\phi^\vepsi(R)| \}  
=\max\{ 0, |\phi^\vepsi(R)| \}.$
If $\phi^\vepsi(R)=0$, then $\phi^\vepsi\equiv 0$. If 
$\phi^\vepsi(R)\neq 0$, then $\phi^\vepsi$ takes its maximum
or minimum value at $r=R$. However, the strong maximum principle  
\cite[Theorem 4, page 7]{Protter_Weinberger67}  
implies that $\phi^\vepsi_r(R)\neq 0$, which contradicts with  
boundary condition $\phi^\vepsi_r(R)= 0$. Hence, $\phi^\vepsi\equiv 0$
which implies $w^\vepsi_1\equiv w^\vepsi_2$. 
\end{proof}


Next, we prove that the existence of nonnegative solutions to 
problem \eqref{moment_wradial}.

\begin{thm}\label{existence_thm}
Suppose $r^{n-1}f\in L^1((0,R))$ and $f\geq 0$ a.e. in $(0,R)$.
Then  there is a nonnegative classical solution to 
problem \eqref{moment_wradial}.
\end{thm}
\begin{proof}
We divide the proof into three steps.

\smallskip
{\em Step 1:}
Let $\psi^0\in C^2([0,R])$ be a nonnegative function 
that  satisfies $\psi^0(0)=\psi^0_r(0)=0$
and $\psi^0_r(R)=\vepsi R^{n-1}$. For example, we may take 
$\psi^0(r)=\frac{\vepsi}n r^n$.
We then define a sequence of functions $\{\psi^k\}_{k\geq 0}$ 
recursively by solving for $k=0,1,2,\cdots$
\begin{subequations}
\label{existence}
\begin{align}\label{existence1}
-\vepsi r^{n-1} \Bigl(\frac{1}{r^{n-1}} \psi^{k+1}_r\Bigr)_r
+\frac{1}{nr^{n(n-1)}} (\psi^k)^{n-1} \psi^{k+1} 
&= L_f(r) 
 \qquad \mbox{in } (0,R),  \\
\psi^{k+1}(0)=\psi^{k+1}_r(0) &=0, \label{existence2} \\
\psi^{k+1}_r(R) &=\vepsi R^{n-1}. \label{existence3} 
\end{align}
\end{subequations}

We first show by induction that for any such sequence satisfying \eqref{existence},
there holds  $\psi^k\geq 0$ in $[0,R]$ for all $k\geq 0$.
Note that $\psi^0\geq 0$ by construction. Suppose that $\psi^k\geq 0$ 
in $[0,R]$.  Since $f\geq 0$ and $f\not \equiv 0$, we have 
\[
-\vepsi r^{n-1}\Bigl(\frac{1}{r^{n-1}} \psi^{k+1}_r\Bigr)_r
+\frac{1}{nr^{n(n-1)}} (\psi^k)^{n-1} \psi^{k+1} > 0 \qquad \mbox{in } (0,R). 
\]
Hence, $\psi^{k+1}$ is a supersolution to the linear differential 
operator appearing in the left-hand side of \eqref{existence1}.  
By the weak maximum principle \cite[Theorem 2, page 329]{Evans98}, we have
\[
\min_{ [0,R]} \psi^{k+1} (r) \geq \min\{0, \psi^{k+1}(0), \psi^{k+1}(R)\}
=\min\{0, \psi^{k+1}(R)\}.
\]

Now suppose that  $\psi^{k+1}(R)< 0$.  Then since $\psi^{k+1}_r(R) =\vepsi R^{n-1}> 0$,
the strong maximum principle \cite[Theorem 4, page 7]{Protter_Weinberger67}
implies that $\psi^{k+1}\equiv\mbox{const}$ in $[0,R]$, which leads to 
a contradiction as $\psi^{k+1}(0)=0$. Thus, we must have
$\psi^{k+1}(R)\geq 0$, and therefore, $\psi^{k+1}\geq 0$ in $[0,R]$. 
By the induction  argument, we conclude that $\psi^k\geq 0$ 
in $[0,R]$ for all $k\geq 0$.

It then follows from the standard theory for linear elliptic equations 
(cf. \cite{Evans98,Gilbarg_Trudinger01}) that \eqref{existence}
has a unique classical solution $\psi^{k+1}$. Hence the $(k+1)$th iterate
$\psi^{k+1}$ is well defined and therefore so is the sequence $\{\psi^{k}\}_{k\geq 0}$.

\smallskip
{\em Step 2:} Next, we shall derive some uniform (in $k$) estimates
for the sequence $\{\psi^k\}_{k\geq 0}$. To this end, we first prove that
$\psi^{k+1}(R)$ can be bounded from above uniformly in $k$.  
Multiplying \eqref{existence1} by $\psi^{k+1}$ and integrating by parts yields
\begin{align}\label{existence4}
&-\vepsi \psi^{k+1}_r(r) \psi^{k+1}(r) \Bigl|_{r=0}^{r=R}
+\vepsi \int_0^R |\psi^{k+1}_r(r)|^2 \,dr \\
&+\vepsi \int_0^R \frac{n-1}{r} \psi^{k+1}_r(r) \psi^{k+1}(r)\, dr 
+ \int_0^R \frac{1}{nr^{n(n-1)}} (\psi^k(r))^{n-1} |\psi^{k+1}(r)|^2\,dr \nonumber \\ 
&\hskip 1in 
= \int_0^R \psi^{k+1}(r) L_f(r) \,dr. \label{existence4} \nonumber
\end{align}

It follows from boundary conditions \eqref{existence2} and \eqref{existence3} that
\begin{equation}\label{existence5a}
-\vepsi \psi^{k+1}_r(r) \psi^{k+1}(r) \Bigl|_{r=0}^{r=R}=-\vepsi^2 R^{n-1}\, \psi^{k+1}(R).
\end{equation}
Integrating by parts gives us
\begin{align}\label{existence5}
\int_0^R \frac{1}{r} \psi^{k+1}_r(r) \psi^{k+1}(r)\, dr
&=\frac{1}{2r} \bigl(\psi^{k+1}(r)\bigr)^2\Bigr|_{r=0}^{r=R} 
+\int_0^R \frac{1}{2r^2}\bigl(\psi^{k+1}(r)\bigr)^2\, dr \\
&=\frac{1}{2R} \bigl(\psi^{k+1}(R)\bigr)^2
+ \int_0^R \frac{1}{2r^2} \bigl(\psi^{k+1}(r)\bigr)^2\, dr. \nonumber
\end{align}
By Schwarz, Poincar\'e, and Young's inequalities, we get
\begin{align}\label{existence6}
&\int_0^R \psi^{k+1}(r) L_f(r) \,dr 
\leq \frac{\vepsi}2 \int_0^R |\psi^{k+1}_r(r)|^2 \,dr
+ \frac{C_1^2}{2\vepsi} \int_0^R \bigl(L_f(r)\bigr)^2\, dr
\end{align}
for some positive constant $C_1=C_1(R)$.
Combining \eqref{existence4}--\eqref{existence6} we obtain
\begin{align}\label{existence7a} 
&-2\vepsi^2 R^{n-1}\, \psi^{k+1}(R) + \vepsi \int_0^R |\psi^{k+1}_r(r)|^2 \,dr
+\frac{\vepsi(n-1)}{R}\bigl(\psi^{k+1}(R)\bigr)^2 \\
&+ \int_0^R \frac{\vepsi(n-1)}{r^2} \bigl(\psi^{k+1}(r)\bigr)^2\, dr
+\int_0^R \frac{2}{nr^{n(n-1)}} (\psi^k(r))^{n-1} |\psi^{k+1}(r)|^2\,dr 
\nonumber\\
&\hskip 1.8in
\leq \frac{C_1^2}{\vepsi} \int_0^R \bigl(L_f(r)\bigr)^2\, dr.
\nonumber
\end{align}
Let 
$z:=\psi^{k+1}(R),\
b:=\frac{2\vepsi R^n}{n-1},$
and 
$c:=\frac{C_1^2R^2}{\vepsi^2(n-1)}\bigl(L_f(R)\bigr)^2.$
Then from \eqref{existence7a} we have $z^2-bz-c\le 0$,
which in turn implies that
\[
z_1\leq z \leq z_2,\quad\mbox{where}\quad  z_1=\frac{b-\sqrt{b^2+4c}}2,\quad
z_2=\frac{b+\sqrt{b^2+4c}}2.
\]
Since $-z_2< z_1$, the above inequality then infers that
$|z|\leq z_2$. Thus, there exists a positive constant
$C_2=C_2(R, L_f)$ such that
$\bigl|\psi^{k+1}(R)\bigr|\leq z_2\leq \frac{C_2}{\vepsi}.$
Substituting this bound into the first term on the
left-hand side of \eqref{existence7a} we also get
\begin{align}\label{existence8}
\vepsi \int_0^R |\psi^{k+1}_r(r)|^2 \,dr 
&+\int_0^R \frac{\vepsi(n-1)}{r^2} \bigl(\psi^{k+1}(r)\bigr)^2\, dr \\
& +\int_0^R \frac{2}{nr^{n(n-1)}} \bigl(\psi^k(r)\bigr)^{n-1}|\psi^{k+1}(r)|^2\,dr
\nonumber\\
&\leq C_3:=\frac{C_1^2R}{\vepsi} \bigl(L_f(R)\bigr)^2 + 2\vepsi^2 R^{n-1} C_2. 
\nonumber
\end{align}

Now using the pointwise estimate for linear elliptic equations
\cite[Theorem 3.7]{Gilbarg_Trudinger01} we have 
\begin{equation}\label{existence9}
\max_{[0,R]} |\psi^{k+1}(r)| \leq \Bigl(\frac{C_2}{\vepsi}  
+ \frac{L_f(R)}{\vepsi}\Bigr).
\end{equation}

Next, we show that $\psi^{k+1}_r$ is also uniformly bounded (in $k$) in $[0,R]$.
To this end, integrating \eqref{existence1} over $(0,r)$ after multiplying
it by $r^{n(n-1)}$, and integrating by parts twice in the first term yields
\begin{align}\label{existence10}
\psi^{k+1}_r(r)&= 
-\frac{(n^2-1)[n(n-1)-1]}{r^{n(n-1)}} \int_0^r s^{n(n-1)-2} \psi^{k+1}(s)\,ds 
\\
&\hspace{-1cm}\, +\frac{n^2-1}{r} \psi^{k+1}(r)
+\frac{1}{\vepsi nr^{n(n-1)}} \int_0^r  \bigl(\psi^k(s)\bigr)^{n-1}
\psi^{k+1}(s)\,ds \nonumber \\ 
&\quad\, 
-\frac{1}{\vepsi r^{n(n-1)}}  \int_0^r s^{n(n-1)} L_f(s) \,ds
\nonumber
\end{align}
for all $r\in (0,R)$.
Using L'H\^opital's rule it is easy to check that the limit as
$r\to 0^+$ of each term on the right-hand side of \eqref{existence10}
is zero. Hence, each term is bounded in a neighborhood of $r=0$. 
Moreover, on noting that $\psi^k\geq 0$, by Schwarz inequality we have
\begin{align}\label{existence11}
&\int_0^r  \bigl(\psi^k(s)\bigr)^{n-1} \psi^{k+1}(s)\,ds \\
&\leq \Bigl(\int_0^r s^{n(n-1)}  \bigl(\psi^k(s)\bigr)^{n-1} \,ds\Bigr)^{\frac12} 
\Bigl(\int_0^r \frac{1}{s^{n(n-1)}} \bigl(\psi^k(s)\bigr)^{n-1} |\psi^{k+1}(s)|^2\,ds\Bigr)^{\frac12}.
\nonumber
\end{align}

Now in view of \eqref{existence8}--\eqref{existence11} we conclude 
that there exists a positive constant $C_4=C_4(R,L_f)$ such that
\begin{equation}\label{existence12}
\max_{[0,R]} |\psi^{k+1}_r(r)| \leq \frac{C_4}{\vepsi^{\frac{n+2}2}}.
\end{equation}

By \eqref{existence1} we get
\begin{align}\label{existence13}
\psi^{k+1}_{rr}(r)=\frac{1}{r} \psi^{k+1}_r(r) 
&+\frac{1}{\vepsi nr^{n(n-1)}} \bigl(\psi^k(r)\bigr)^{n-1} \psi^{k+1}(r) \\ 
& -\frac{1}{\vepsi} L_f(r) \qquad\forall r\in (0,R).\nonumber  
\end{align}
Again, using L'H\^opital's rule and \eqref{moment5a_radial}
it is easy to check that the limit as
$r\to 0^+$ of each term on the right-hand side of \eqref{existence13}
exists, and therefore, each term is bounded in a neighborhood of $r=0$.
Hence, it follows from \eqref{existence9} and \eqref{existence12} that
there exists a positive constant $C_5=C_5(R,L_f)$ such that
\begin{equation}\label{existence14}
\max_{[0,R]} |\psi^{k+1}_{rr}(r)| \leq \frac{C_5}{\vepsi^{n+1}}.
\end{equation}

To summarize, we have proved that 
$\|\psi^{k+1}\|_{C^j([0,R])}\leq C(\vepsi,R,n,L_f)$ for $j=0,1,2$ 
and the bounds are independent of $k$.  Clearly, by a
simple induction argument we conclude that these estimates hold for 
all $k\geq 0$.

\smallskip
{\em Step 3:} 
Since $\|\psi^k\|_{C^2([0,R])}$ is uniformly bounded in $k$, then 
both $\{\psi^k\}_{k\geq 0}$ and $\{\psi^k_r\}_{k\geq 0}$
are uniformly equicontinuous. It follows from
Arzela-Ascoli compactness theorem (cf. \cite[page 635]{Evans98}) that
there is a subsequence of $\{\psi^k\}_{k\geq 0}$ (still denoted by the
same notation) and $\psi\in C^2([0,R])$ such that 
\begin{alignat*}{2}
\psi^k &\longrightarrow \psi\ \text{ and }\ \psi^k_r \longrightarrow \psi_r
&&\quad\mbox{uniformly in every compact set $E\subset (0,R)$ as } k\to \infty.
\end{alignat*}

Testing equation \eqref{existence1} with an arbitrary function
$\chi\in C^1_0((0,R))$ yields
\begin{align*}
&\vepsi\int_0^R \psi^{k+1}_r(r) \chi_r(r) \,dr  
+\vepsi\int_0^R \frac{n-1}{r} \psi^{k+1}_r(r) \chi(r) \,dr \\
&\quad
+\int_0^R \frac{1}{nr^{n(n-1)}}\bigl(\psi^k(r)\bigr)^{n-1} \psi^{k+1}(r) \chi(r)\,dr
=\int_0^R L_f(r) \chi(r)\, dr. 
\end{align*}
Setting $k\to \infty$ and using the Lebesgue Dominated Convergence 
Theorem, we obtain
\begin{align}\label{existence15}
\vepsi\int_0^R \psi_r(r) \chi_r(r) \,dr  
&+\vepsi\int_0^R \frac{n-1}{r} \psi_r(r) \chi(r) \,dr \\
&+\int_0^R \frac{1}{nr^{n(n-1)}} \bigl(\psi(r)\bigr)^n \chi(r)\,dr
=\int_0^R L_f(r) \chi(r)\, dr. \nonumber
\end{align}
Since $\psi\in C^2([0,R])$, we are able to integrate by parts in the first term on the 
left-hand side of \eqref{existence15}, yielding 
\[
\int_0^R \Bigl[ -\vepsi \psi_{rr}(r) +\frac{\vepsi(n-1)}{r} \psi_r(r) 
+  \frac{1}{nr^{n(n-1)}} \bigl(\psi(r)\bigr)^n - L_f(r) \Bigr] \chi(r)\, dr=0
\]
for all $\chi\in C^1_0((0,R))$. This then implies that
\[
-\vepsi \psi_{rr}(r) +\frac{\vepsi(n-1)}{r} \psi_r(r) 
+ \frac{1}{nr^{n(n-1)}} \bigl(\psi(r)\bigr)^n
- L_f(r) =0 \qquad\forall r\in (0,R),
\]
that is,
\[
-\vepsi r^{n-1} \Bigl(\frac{1}{r^{n-1}} \psi_r(r)\Bigr)_r  
+ \frac{1}{nr^{n(n-1)}} \bigl(\psi(r)\bigr)^n = L_f(r)\qquad\forall r\in (0,R).
\]
Thus, $\psi$ satisfies \eqref{moment1c_radial} pointwise in $(0,R)$.

Finally, it is clear that $\psi \geq 0$ in $[0,R]$, and it follows easily 
from \eqref{existence2} and \eqref{existence3} that 
$\psi(0)=\psi_r(0) =0$ 
and $\psi_r(R) =\vepsi R^{n-1}.$
So we have demonstrated that $\psi\in C^2([0,R])$ is a nonnegative
classical solution to problem \eqref{moment1c_radial}--\eqref{moment4a_radial}.  
The proof is complete.
\end{proof}

\begin{remark}
We note that the a priori estimates derived in the proof are not sharp
in $\vepsi$. Better estimates will be obtained (and needed) in the
next section after the positivity of $\Del u^\vepsi$ is established.
\end{remark}

The above proof together with the uniqueness theorem, 
Theorem \ref{uniqueness_thm}, and the strong
maximum principle immediately give the following corollary.

\begin{cor}\label{existence_cor}
Suppose $r^{n-1}f\in L^1((0,R))$ and $f\geq 0$ a.e. in $(0,R)$, then 
there exists a unique nonnegative classical solution $w^\vepsi$ to
problem \eqref{moment_wradial}. Moreover,
$w^\vepsi>0$ in $(0,R)$, $w^\vepsi \in C^3((0,R))$ if $f\in C^0((0,R))$, and 
$w^\vepsi$ is smooth provided that $f$ is smooth.
\end{cor}

Recall that $w^\vepsi=r^{n-1}u^\vepsi_r$ where $u^\vepsi$ and $w^\vepsi$ are
solutions of \eqref{moment_radial} and
\eqref{moment_wradial} respectively. Let $w^\vepsi$ be the
unique solution to \eqref{moment_wradial}
as stated in Corollary \ref{existence_cor} and define
\begin{equation}\label{existence19}
u^\vepsi(r):=g(R) - \int_r^R \frac{1}{s^{n-1}} w^\vepsi(s)\, ds
\qquad\forall r\in (0,R).
\end{equation}

We now show that $u^\vepsi$ is a unique monotone increasing classical 
solution of problem \eqref{moment_radial}.

\begin{thm}\label{momnet1_existence}
Suppose $f\in C^0((0,R))$ and $f\geq 0$ in $(0,R)$, then problem 
\eqref{moment_radial}
has a unique monotone increasing classical solution.  Moreover, 
$u^\vepsi$ is smooth provided that $f$ is smooth.
\end{thm}

\begin{proof}
By direct calculations one can easily show that $u^\vepsi$ defined
by \eqref{existence19} satisfies \eqref{moment_radial}.
Since $u^\vepsi_r> 0$ in $(0,R)$, then $u^\vepsi$ is a monotone increasing 
function. Hence, the existence is shown. 

To show uniqueness, we notice that $u^\vepsi$ is 
a monotone increasing classical solution
of problem \eqref{moment_radial}
if and only if $w^\vepsi$ is a nonnegative classical solution of 
problem \eqref{moment_wradial}.
Hence, the uniqueness of \eqref{moment_radial}
follows from the uniqueness of \eqref{moment_wradial}.
\end{proof}

\section{Convexity of vanishing moment approximations} \label{section-4}
The goal of this section is to analyze the convexity of 
the solution $u^\vepsi$ whose existence is proved in 
Theorem \ref{momnet1_existence}. We shall prove that $u^\vepsi$
is strictly convex either in $(0,R)$ or in $(0,R-c_0\vepsi)$
for some $\vepsi$-independent positive constant $c_0$. 
From calculations in Section \ref{section-2} we know that
$D^2u^\vepsi$ only has two distinct eigenvalues: $\lambda_1=u^\vepsi_{rr}$
(with multiplicity $1$) and $\lambda_2=\frac{1}{r}u^\vepsi_r$
(with multiplicity $n-1$).  We have proved that $\lambda_2\geq 0$ 
in $(0,R)$, so it is necessary to show $\lambda_1\geq 0$ in $(0,R)$ 
or in $(0,R-c_0\vepsi)$. In addition, in this section we   
derive some sharp uniform (in $\vepsi$) {\em a priori} estimates 
for the vanishing moment approximations $u^\vepsi$ which will play
an important role not only for establishing the convexity property for $u^\vepsi$,
but also for proving the convergence of $u^\vepsi$ in the next section.

First, we have the following positivity result for $\Del u^\vepsi$.

\begin{thm}\label{convexity_thm1}
Let $u^\vepsi$ be the unique monotone increasing classical solution
of problem \eqref{moment_radial}
and define $w^\vepsi :=r^{n-1} u^\vepsi_r$. Then 
\begin{itemize}
\item[{\rm (i)}] $w^\vepsi_r> 0$ in $(0,R)$.  Consequently, $\Del u^\vepsi >0$ 
in $(0,R)$, for all $\vepsi>0$.
\item[{\rm (ii)}] For any $r_0\in (0,R)$, there exists an $\vepsi_0>0$ such that
$w^\vepsi_r>\vepsi r^{n-1}$ and $\Del u^\vepsi>\vepsi$ in $(r_0,R)$ for 
$\vepsi \in (0,\vepsi_0)$.
\end{itemize}
\end{thm}

\begin{proof} We split the proof into two steps.

{\em Step 1:} Since $u^\vepsi$ is monotone increasing and differentiable,
we have $u^\vepsi_r\geq 0$ in $[0,R]$. From the derivation of 
Section \ref{section-2} we know that 
$w^\vepsi:=r^{n-1}u^\vepsi_r$ is the unique nonnegative 
classical solution of \eqref{moment_wradial}.
Let $\varphi^\vepsi := w^\vepsi_r$.
By the definition of the Laplacian $\Del$ we have
\begin{equation}\label{convexity1}
\varphi^\vepsi= w^\vepsi_r =r^{n-1} u^\vepsi_{rr}
+ (n-1) r^{n-2} u^\vepsi_r =r^{n-1}\Del u^\vepsi.
\end{equation}
So $\varphi^\vepsi> 0$ in $(0,R)$ infers $\Del u^\vepsi > 0$ 
in $(0,R)$. 
To show $\varphi^\vepsi> 0$, we differentiate 
\eqref{moment1c_radial} with respect to $r$ to get
\begin{align*}
-\vepsi w^\vepsi_{rrr} 
&+\vepsi(n-1) r^{n-2} \Bigl(\frac{1}{r^{n-1}} w^\vepsi_r \Bigr)_r 
+\Bigl[ \frac{\vepsi(n-1)(n-2)}{r^2} + \frac{(w^\vepsi)^{n-1}}{r^{n(n-1)}}
\Bigr] w^\vepsi_r \\
&\quad -\frac{n-1}{r^{(n-1)^2+n}} (w^\vepsi)^n 
=r^{n-1} f(r) \qquad\mbox{in } (0,R).
\end{align*}
From \eqref{moment1c_radial}, we have
$
\vepsi r^{n-2}\Bigl(\frac{1}{r^{n-1}} w^\vepsi_r \Bigr)_r
=\frac{1}{nr^{(n-1)^2+n}} \bigl(w^\vepsi\bigr)^n
-\frac{1}{r} L_f.
$
Combining the above two equations yields
\begin{align}\label{convexity2a}
-\vepsi w^\vepsi_{rrr} 
+ \Bigl[ \frac{\vepsi(n-1)(n-2)}{r^2} + \frac{(w^\vepsi)^{n-1}}{r^{n(n-1)}}
\Bigr] w^\vepsi_r 
& =r^{n-1}f + \frac{n-1}{r} L_f + \frac{(n-1)^2}{nr^{(n-1)^2+n}} 
\bigl(w^\vepsi\bigr)^n. 
\end{align}
Substituting $w^\vepsi_r = \varphi^\vepsi$ into the above equation we get 
\begin{align}\label{convexity2}
-\vepsi \varphi^\vepsi_{rr} 
&+\Bigl[ \frac{\vepsi(n-1)(n-2)}{r^2} + \frac{(w^\vepsi)^{n-1}}{r^{n(n-1)}}
\Bigr] \varphi^\vepsi \\
&\quad 
=r^{n-1}f + \frac{n-1}{r} L_f + \frac{(n-1)^2}{nr^{(n-1)^2+n}} 
\bigl(w^\vepsi\bigr)^n \geq 0
\nonumber
\end{align}
since $f, L_f, w^\vepsi\geq 0$ in $(0,R)$.
This means that $\varphi^\vepsi$ is a supersolution to a linear uniformly
elliptic differential operator. By the weak maximum principle 
we obtain (cf. \cite[page 329]{Evans98})
\[
\min_{[0,R]} \varphi^\vepsi (r) 
\geq \min\{0, \varphi^\vepsi(0), \varphi^\vepsi(R)\}
=\min\{0,0,R^{n-1} \vepsi\}=0.
\]
Here we have used the fact that 
$\varphi^\vepsi(R)=R^{n-1} \Del u^\vepsi(R)= R^{n-1} \vepsi$.
Hence, $\varphi^\vepsi\geq 0$ in $[0,R]$, so 
$\Del u^\vepsi\geq 0$ in $[0,R]$.

It follows from the strong maximum principle
(cf.\,\cite[Theorem 4, page 7]{Protter_Weinberger67}) that
$\varphi^\vepsi$ cannot attain its nonpositive minimum value $0$ at
any point in $(0,R)$. Therefore, $\varphi^\vepsi > 0$ in $(0,R)$,  
which implies that $\Del u^\vepsi>0$ in $(0,R)$. So assertion (i) holds.

\smallskip
{\em Step 2:} To show (ii), let $\psi^\vepsi:= w^\vepsi_r-\vepsi r^{n-1} 
=r^{n-1}(\Del u^\vepsi-\vepsi)$.  Using the identities 
$w^\vepsi_r = \psi^\vepsi + \vepsi r^{n-1}$ and
$w^\vepsi_{rrr}= \psi^\vepsi_{rr} + \vepsi (n-1)(n-2) r^{n-3},$
we rewrite \eqref{convexity2a} as 
\begin{align}\label{convexity2b}
-\vepsi \psi^\vepsi_{rr} 
&+\Bigl[ \frac{\vepsi(n-1)(n-2)}{r^2} + \frac{(w^\vepsi)^{n-1}}{r^{n(n-1)}}
\Bigr] \psi^\vepsi \\
&=r^{n-1}f + \frac{n-1}{r} L_f 
+ \frac{(w^\vepsi)^{n-1}[(n-1)^2 w^\vepsi -\vepsi n r^n ]}{nr^{(n-1)^2+n}}
\qquad \mbox{in } (0,R). \nonumber
\end{align}
Hence, $\psi^\vepsi$ satisfies a linear uniformly elliptic equation.

Now, on noting that $w^\vepsi\geq 0$ by (i), 
for any $r_0\in (0,R)$ (i.e., $r_0$ is away from $0$), it is easy to
see that there exists an $\vepsi_1>0$ such that the right-hand side 
of \eqref{convexity2b} is nonnegative in $(r_0,R)$ for all $\vepsi\in (0,\vepsi_1)$.
Hence, $\psi^\vepsi$ is a supersolution in $(r_0,R)$ to the uniformly elliptic 
operator on the right-hand side of \eqref{convexity2b}. By the weak maximum principle
we have (cf. \cite[page 329]{Evans98})
\[
\min_{[r_0,R]} \psi^\vepsi (r) 
\geq \min\{0, \psi^\vepsi(r_0), \psi^\vepsi(R)\}
=\min\{0,\Del u^\vepsi(r_0)-\vepsi, 0\}.
\]
Again, here we have used the fact that $\Del u^\vepsi(R)= \vepsi$.

Since $\Del u^\vepsi(r_0)>0$, we can choose 
$\vepsi_0=\min\{\vepsi_1, \frac12 \Del u^\vepsi(r_0)\}$.
We then have $\psi^\vepsi(r_0)=\Del u^\vepsi(r_0)-\vepsi \geq \frac12 \Del u^\vepsi(r_0)>0$
for $\vepsi\in (0,\vepsi_0)$. Thus, $\min_{[r_0,R]} \psi^\vepsi (r) \geq 0$
for $\vepsi\in (0,\vepsi_0)$. Therefore, $w^\vepsi_r\geq \vepsi r^{n-1}$,
and consequently, $\Del u^\vepsi\geq \vepsi$ in $[r_0,R]$ for $\vepsi\in (0,\vepsi_0)$.

Finally, an application of the strong maximum principle
(cf. \cite[Theorem 4, page 7]{Protter_Weinberger67}) yields
that $w^\vepsi_r> \vepsi r^{n-1}$.  Hence, $\Del u^\vepsi> \vepsi$ in $(r_0,R)$
for $\vepsi\in (0,\vepsi_0)$. The proof is complete.
\end{proof}

\begin{remark}
The proof also shows that $\vepsi_0$ decreases (resp. increases) as $r_0$ decreases
(resp. increases), and $v^\vepsi:=\Del u^\vepsi$ takes its minimum value
$\vepsi$ in $[r_0,R]$ at the right end of the interval $r=R$.
\end{remark}

With help of the positivity of $\Del u^\vepsi$, we can derive 
some better uniform estimates (in $\vepsi$) for $w^\vepsi$ and $u^\vepsi$.

\begin{thm}\label{fine_estimates}
Suppose $f\in C^0((0,R))$ and $f\geq 0$ in $(0,R)$. Let $u^\vepsi$ be
the unique monotone increasing classical solution to problem
\eqref{moment1_radial}--\eqref{moment4_radial}.  Define
$w^\vepsi=r^{n-1} u^\vepsi_r$ and $v^\vepsi= \Del u^\vepsi=u^\vepsi_{rr}
+\frac{n-1}{r} u^\vepsi_r$.  Then there holds the following estimates
for sufficiently small $\vepsi>0$:
\begin{alignat*}{2}
&\mbox{\rm (i)}  &&\,\|u^\vepsi\|_{C^0([0,R])} 
+\int_0^R |u^\vepsi_r|^n\, dr \leq C_0, \\
&\mbox{\rm (ii)} &&\,\|u^\vepsi\|_{C^1([0,R])} 
+ \|w^\vepsi\|_{C^0([0,R])} \leq C_1, \\
&\mbox{\rm (iii)} &&\,\|w^\vepsi_r\|_{C^0([0,R])}\leq \frac{C_2}{\vepsi}, \\
&\mbox{\rm (iv)}  &&\|v^\vepsi\|_{C^0([r_0,R])}\leq \frac{C_3}{\vepsi r_0^{n-1}} 
\quad\forall 0<r_0\leq R, \\
&\mbox{\rm (v)} &&\,\|v^\vepsi_r\|_{C^0([r_0,R])}
\leq \frac{C_4}{\vepsi r_0^{(n-1)^2}} \quad\forall 0<r_0\leq R, \\
&\mbox{\rm (vi)} &&\, \int_0^R |w^\vepsi_r(r)|^2\, dr
+ \int_0^R r^{2(n-1)} |v^\vepsi(r)|^2  \,dr \leq \frac{C_5}{\vepsi}, \\
&\mbox{\rm (vii)} &&\,\vepsi \int_0^R r^{n-2-\alpha} |v^\vepsi(r)|^2 \, dr 
+ \int_0^R \frac{1}{r^\alpha}(u^\vepsi_r(r))^n v^\vepsi(r)\, dr 
\leq \frac{C_6}{\vepsi} \quad\forall \alpha< n-1, 
\end{alignat*}
\begin{alignat*}{2}
&\mbox{\rm (viii)} &&\,\vepsi \int_0^R r^{n-1} |v^\vepsi_r(r)|^2 \, dr 
+ \int_0^R (u^\vepsi_r(r))^{n-1} |v^\vepsi(r)|^2\, dr 
\leq  \frac{C_7}{\vepsi} \quad\mbox{for } n\geq 3, \\
&\mbox{\rm (ix)} &&\,\,\vepsi\int_0^R r^{2-\alpha} |v^\vepsi_r(r)|^2 \, dr 
+ \int_0^R r^{1-\alpha} u^\vepsi_r(r) |v^\vepsi(r)|^2\, dr\leq \frac{C_8}{\vepsi} 
\quad\mbox{for } n=2,\, \alpha<1,
\end{alignat*}
where $C_j=C_j(R,f,n)>0$ for $j=0,1,2,\cdots,8$ are $\vepsi$-independent
positive constants.
\end{thm}

\begin{proof}
We divide the proof into five steps

{\em Step 1}: Since $u^\vepsi$ is monotone increasing,
\begin{align}\label{v_eqn_1}
\max_{[0,R]} u^\vepsi(r) \leq u^\vepsi(R)=g(R).
\end{align}
%
On noting that $w^\vepsi$ satisfies equation \eqref{moment1c_radial},
integrating \eqref{moment1c_radial} over $(0,R)$ and using integration by
parts on the first term on the left-hand side yields
\begin{align*} 
-\vepsi w^\vepsi_r(R) + \vepsi (n-1) \int_0^R \frac{1}{r} w^\vepsi(r)\, dr 
+ \frac{1}{n} \int_0^R \Bigl[\frac{w^\vepsi(r)}{r^{n-1}} \Bigr]^n\, dr
= \int_0^R L_f(r)\, dr.
\end{align*}

Because $w^\vepsi_r(R)=\vepsi R^{n-1}$ and $w^\vepsi\geq 0$, the above
equation and the relation $w^\vepsi=r^{n-1} u^\vepsi_r$ imply that
\begin{align} \label{v_eqn_7}
\int_0^R \Bigl|\frac{w^\vepsi(r)}{r^{n-1}} \Bigr|^n\, dr
= \int_0^R \bigl|u^\vepsi_r(r) \bigr|^n\, dr
\leq n R\bigl[L_f(R) + \vepsi^2 R^{n-2} \bigr]. 
\end{align}
It then follows from \eqref{v_eqn_1}, \eqref{v_eqn_7} and \eqref{existence19} that
\begin{align} \label{v_eqn_8}
g(R)-nR\bigl[ L_f(R) + \vepsi^2 R^{n-2} \bigr]^{\frac{1}{n}} 
\leq u^\vepsi(r) \leq g(R)
\qquad\forall r\in [0,R].
\end{align}
Hence, $u^\vepsi$ is uniformly bounded (in $\vepsi$) in $[0,R]$, and (i) holds.

\smallskip
{\em Step 2}: Let 
$v^\vepsi:=\Del u^\vepsi = u^\vepsi_{rr} + \frac{n-1}{r} u^\vepsi_r
= \frac{1}{r^{n-1}} \bigl(r^{n-1} u^\vepsi_r\bigr)_r.$
By \eqref{moment1_radial} we have
\begin{align}\label{v_eqn_0}
-\vepsi \bigl( r^{n-1} v^\vepsi_r \bigr)_r  
+ \frac{1}{n} \bigl((u^\vepsi_r)^n\bigr)_r =r^{n-1} f \qquad \mbox{in } (0,R).
\end{align}
It was proved in the previous theorem that $v^\vepsi> \vepsi$ 
in $(\frac{R}2,R)$ for sufficiently small $\vepsi>0$,
and it takes its minimum value $\vepsi$ at $r=R$.  Hence we have
$v^\vepsi_r(R)\leq 0$. We note that this is the only place in the proof 
where we may need to require $\vepsi$ to be sufficiently small.
Integrating \eqref{v_eqn_0} over $(0,R)$ yields
\begin{align*} 
-\vepsi r^{n-1} v^\vepsi_r \Bigr|_{r=0}^{r=R}  
+ \frac{1}{n} (u^\vepsi_r)^n\Bigr|_{r=0}^{r=R} =L_f(R). 
\end{align*}
We then have
$\bigl(u^\vepsi_r(R)\bigr)^n
= n L_f(R) +\vepsi n R^{n-1} v^\vepsi_r(R) \leq  nL_f(R),$
and therefore,
\begin{align} \label{v_eqn_2}
u^\vepsi_r(R)=
\bigl|u^\vepsi_r(R)\bigr|
\leq  \bigl(nL_f(R)\bigr)^{\frac{1}{n}}.
\end{align}
Here, we have used boundary condition \eqref{moment3_radial} and
the fact that $v^\vepsi_r(R)\leq 0$ and $u^\vepsi_r\geq 0$.

By the definition of $w^\vepsi(r):=r^{n-1} u^\vepsi_r(r)$, we have
\begin{align} \label{v_eqn_3}
w^\vepsi(R)=
\bigl|w^\vepsi(R)\bigr|
\leq R^{n-1} \bigl|u^\vepsi_r(R)\bigr|
\leq R^{n-1} \bigl(nL_f(R)\bigr)^{\frac{1}{n}}.
\end{align}
Using the identity
$v^\vepsi(r)=\Del u^\vepsi(r)=u^\vepsi_{rr}(r) +\frac{n-1}{r} u^\vepsi_r(r),$
we obtain
\[
u^\vepsi_{rr}(R)=\Del u^\vepsi(R)-\frac{n-1}{R} u^\vepsi_r(R)
=\vepsi-\frac{n-1}{R} u^\vepsi_r(R). 
\]
Therefore,
\begin{align} \label{v_eqn_4}
\bigl|u^\vepsi_{rr}(R)\bigr|
\leq \vepsi + \frac{n-1}{R} \bigl|u^\vepsi_r(R)\bigr|
\leq \vepsi + \frac{n-1}{R}  \bigl(nL_f(R)\bigr)^{\frac{1}{n}}.
\end{align}

\smallskip
{\em Step 3}: From Theorem \ref{convexity_thm1} we have that 
$w^\vepsi_r(r)\geq 0$ in $(0,R)$, and hence,
$w^\vepsi$ is monotone increasing. Consequently,
\begin{align} \label{v_eqn_5}
\max_{[0,R]} w^\vepsi(r)
= \max_{[0,R]} \bigl|w^\vepsi(r)\bigr|
\leq w^\vepsi(R) \leq R^{n-1} \bigl(nL_f(R)\bigr)^{\frac{1}{n}}.
\end{align}

Evidently, \eqref{v_eqn_5} and the relation $w^\vepsi(r)=r^{n-1}u^\vepsi_r(r)$
as well as $\lim_{r\to 0^+} u^\vepsi_r(r)=0$ imply that there exists
$r_0>0$ such that
\begin{align} \label{v_eqn_6}
\max_{[0,R]} u^\vepsi_r(r)
= \max_{[0,R]} \bigl|u^\vepsi_r(r)\bigr|
\leq \frac12 + \Bigl(\frac{R}{r_0} \Bigr)^{n-1} \bigl(nL_f(R)\bigr)^{\frac{1}{n}}.
\end{align}
Hence, (ii) holds.  In addition, since $w^\vepsi_r$ satisfies the linear elliptic equation 
\eqref{convexity2a}, by the pointwise estimate for linear elliptic equations
\cite[Theorem 3.7]{Gilbarg_Trudinger01}, we have
\begin{align}\label{v_eqn_6a}
\max_{[0,R]} |w^\vepsi_r(r)| \leq \vepsi R^{n-1} 
+\frac{1}{\vepsi} \Bigl( R^{n-1}\|f\|_{L^\infty} &+ (n-1) \|r^{-1}L_f\|_{L^\infty}\\
& +\frac{(n-1)^2}{n} \|r^{-1} (u^\vepsi_r)^n\|_{L^\infty} \Bigr).  \nonumber
\end{align}
Since $w^\vepsi_r=r^{n-1}\Del u^\vepsi=:r^{n-1}v^\vepsi$, it follows 
from \eqref{v_eqn_6a} that for any $r_0 >0$ there holds
\begin{align}\label{v_eqn_6b}
\max_{[r_0,R]} |v^\vepsi(r)|
&=\max_{[r_0,R]} |\Del u^\vepsi(r)| \\
&\leq \vepsi \Bigl(\frac{R}{r_0}\Bigr)^{n-1} 
+\frac{1}{\vepsi r_0^{n-1}} \Bigl(R^{n-1} \|f\|_{L^\infty} 
+ (n-1) \|r^{-1}L_f\|_{L^\infty} \nonumber \\
&\hskip 0.8in
+\frac{(n-1)^2}{n} \|r^{-1} (u^\vepsi_r)^n\|_{L^\infty} \Bigr).  \nonumber
\end{align}
Thus, (iii) and (iv) are true.

Integrating \eqref{v_eqn_0} over $(0,r)$ yields
\begin{align}\label{v_eqn_6c}
-\vepsi r^{n-1} v^\vepsi_r + \frac{1}{n} (u^\vepsi_r)^n
=L_f \qquad\mbox{in } (0,R).
\end{align}
By \eqref{v_eqn_6c} and \eqref{v_eqn_6} we conclude that
for any $r_0>0$ there holds
\begin{align} \label{v_eqn_6d}
\max_{[r_0,R]} \bigl|v^\vepsi_r(r)\bigr|
=\max_{[r_0,R]} \bigl| (\Del u^\vepsi(r))_r\bigr|
\leq \frac{1}{\vepsi}
\Bigl( 1+\Bigl(\frac{R}{r_0} \Bigr)^{n(n-1)} \Bigr) \frac{L_f(R)}{r_0^{n-1}}.  
\end{align}
Therefore the estimate (v) holds.

\smallskip
{\em Step 4}: Testing \eqref{moment1c_radial} with $w^\vepsi$ 
and integrating by parts twice on the first term on the left-hand side, we get
\begin{align*} 
&-\vepsi^2 R^{n-1} w^\vepsi(R) + \frac{\vepsi}2 \int_0^R |w^\vepsi_r(r)|^2\, dr
+\frac{\vepsi (n-1)}{2R} [w^\vepsi(R)]^2 \\
&+ \int_0^R \frac{\vepsi (n-1)}{2r^2} |w^\vepsi(r)|^2\, dr
+\int_0^R \frac{1}{nr^{n(n-1)}} |w^\vepsi(r)|^{n+1}\, dr
=\int_0^R L_f(r) w^\vepsi(r)\, dr.
\end{align*}
Combing the above equation and \eqref{v_eqn_5} we obtain
\begin{align*}
\frac{\vepsi}2 \int_0^R |w^\vepsi_r(r)|^2\, dr
&+\int_0^R \frac{\vepsi (n-1)}{2r^2} |w^\vepsi(r)|^2\, dr
+\int_0^R \frac{1}{nr^{n(n-1)}} |w^\vepsi(r)|^{n+1}\, dr \\
&\hspace{1cm}\leq R\bigl[\vepsi^2 R^{n-2}+ L_f(R)\bigr]\,\bigl(nL_f(R)\bigr)^{\frac{1}{n}}. 
\nonumber
\end{align*}
Consequently, we have
\begin{align*}
&\frac{\vepsi}2 \int_0^R |r^{n-1}\Del u^\vepsi(r)|^2\, dr
+\frac{\vepsi (n-1)}{2} \int_0^R r^{2(n-2)}  |u^\vepsi_r(r)|^2\, dr \\
&\qquad\qquad
+\frac{1}{n}\int_0^R r^{n-1} |u^\vepsi_r(r)|^{n+1}\, dr
\leq R\bigl[\vepsi^2 R^{n-2}+ L_f(R)\bigr]\,\bigl(nL_f(R)\bigr)^{\frac{1}{n}}. 
\nonumber
\end{align*}
Hence, (vi) holds.

\smallskip
{\em Step 5}: 
For any real number $\alpha < n-1$, testing \eqref{v_eqn_6c}
with $r^{-\alpha} v^\vepsi$ and using $ v^\vepsi(R)=\vepsi$ we obtain 
\begin{align} \label{v_eqn_11}
-\frac{\vepsi^3}{2} R^{n-1-\alpha}  
&+\frac{\vepsi (n-1-\alpha)}{2} \int_0^R r^{n-2-\alpha} |v^\vepsi(r)|^2 \, dr \\
&+ \int_0^R \frac{1}{nr^\alpha}(u^\vepsi_r(r))^n v^\vepsi(r)\, dr 
=\int_0^R \frac{1}{r^\alpha} L_f(r) v^\vepsi(r)\, dr.  \nonumber
\end{align} 
On noting that $v^\vepsi\geq 0$, $u^\vepsi_r\geq 0$, and 
$L_f(r) =\int_0^r s^{n-1} f(r)\, dr 
\leq \frac{r^n}{n}\|f\|_{L^\infty},$ 
it follows from \eqref{v_eqn_11} that
\begin{align} \label{v_eqn_12}
&\frac{\vepsi (n-1-\alpha)}{4} \int_0^R r^{n-2-\alpha} |v^\vepsi(r)|^2 \, dr 
+ \frac{1}{n}\int_0^R \frac{1}{r^\alpha}(u^\vepsi_r(r))^n v^\vepsi(r)\, dr  \\
&\hskip 0.7in \leq \frac{\vepsi^3}{2} R^{n-1-\alpha} 
+ \frac{R^{n+3-\alpha} \|f\|_{L^\infty}^2}{\vepsi n^2(n-1-\alpha)(n+3-\alpha)}\, dr
\qquad\forall  \alpha< n-1.
\nonumber
\end{align}
This gives last estimate gives us (vii).

Next, recall that 
$v^\vepsi :=\Del u^\vepsi = u^\vepsi_{rr} + \frac{n-1}{r} u^\vepsi_r.$
Therefore, we can rewrite \eqref{v_eqn_0} as follows:
\begin{align*}
-\vepsi \bigl( r^{n-1} v^\vepsi_r \bigr)_r  
+ (u^\vepsi_r)^{n-1} v^\vepsi 
=r^{n-1} f + \frac{n-1}{r} (u^\vepsi_r)^n
\qquad\mbox{in } (0,R).
\end{align*}
Testing the above equation with $r^\beta v^\vepsi$ for
$\beta >1-n$ and using $ v^\vepsi(R)=\vepsi$, we get
\begin{align*}
&-\vepsi^2 R^{n-1+\beta} v^\vepsi_r(R) 
+\vepsi \int_0^R r^{n-1+\beta} |v^\vepsi_r(r)|^2 \, dr 
+\frac{\vepsi^3\beta R^{n+\beta-2}}{2} \\
&-\frac{\vepsi\beta (n+\beta-2)}{2}\int_0^R r^{n+\beta-3} |v^\vepsi(r)|^2 \, dr
+ \int_0^R r^\beta(u^\vepsi_r(r))^{n-1} |v^\vepsi(r)|^2\, dr \\
&\qquad =\int_0^R \Bigl[ r^{n-1+\beta} f(r) 
+ \frac{n-1}{r^{1-\beta}} (u^\vepsi_r(r))^n\Bigr] v^\vepsi(r)\, dr.
\end{align*} 
It then follows that
\begin{align}\label{v_eqn_13}
&-\vepsi^2 R^{n-1+\beta} v^\vepsi_r(R) 
+\vepsi \int_0^R r^{n-1+\beta} |v^\vepsi_r(r)|^2 \, dr 
+ \frac{\vepsi^3\beta R^{n+\beta-2}}{2} \\
&\qquad-\frac{\vepsi\beta (n+\beta-2)}{2}\int_0^R r^{n+\beta-3} |v^\vepsi(r)|^2 \, dr 
+ \int_0^R r^\beta(u^\vepsi_r(r))^{n-1} |v^\vepsi(r)|^2\, dr
\nonumber \\
&\qquad \leq \frac{\vepsi}{2} \int_0^R r^{n-1+\beta} |v^\vepsi_r(r)|^2 \, dr 
+\frac{1}{2\vepsi} \int_0^R r^{n-1+\beta} |f(r)|^2\, dr \nonumber \\
&\hskip 1in
+ (n-1) \int_0^R \frac{1}{r^{1-\beta}} (u^\vepsi_r(r))^n  v^\vepsi(r)\, dr.
\nonumber
\end{align}  

To continue, we consider the cases $n=2$ and $n>2$ separately.
First, for $n> 2$, 
it follows 
from \eqref{v_eqn_12} with $\alpha=1$ and \eqref{v_eqn_13} with 
$\beta=0$ that 
\begin{align}\label{v_eqn_14}
&\frac{\vepsi}{2} \int_0^R r^{n-1} |v^\vepsi_r(r)|^2 \, dr 
+ \int_0^R (u^\vepsi_r(r))^{n-1} |v^\vepsi(r)|^2\, dr \\
&\qquad \leq \frac{1}{2\vepsi} \int_0^R r^{n-1} |f(r)|^2\, dr 
+ R^{n-2}\Bigl[\frac{\vepsi^3 n(n-1)}{2} 
+ \frac{R^{4} \|f\|_{L^\infty}^2}{\vepsi(n^2-4)}\, dr \Bigr]. 
\nonumber
\end{align}
When $n=2$, we note that $\alpha=1$ is not allowed in \eqref{v_eqn_12}.
Let $\alpha< 1$ be fixed in \eqref{v_eqn_12}, set $\beta=1-\alpha$
in \eqref{v_eqn_13} we obtain
\begin{align}\label{v_eqn_15}
&\frac{\vepsi}{2} \int_0^R r^{2-\alpha} |v^\vepsi_r(r)|^2 \, dr 
+ \int_0^R r^{1-\alpha} u^\vepsi_r(r) |v^\vepsi(r)|^2\, dr \\
&\qquad \leq \frac{1}{2\vepsi} \int_0^R r^{2-\alpha} |f(r)|^2\, dr 
+ 2 R^{1-\alpha} \Bigl[\vepsi^3 
+ \frac{R^4 \|f\|_{L^\infty}^2}{\vepsi(1-\alpha)(5-\alpha)}\, dr \Bigr]. 
\nonumber
\end{align}
Hence, (viii) and (ix) hold. The proof is complete.
\end{proof}

We now state and prove the following convexity result for 
the vanishing moment approximation $u^\vepsi$.

\begin{thm}\label{convexity_thm}
Suppose $f\in C^0((0,R))$ and there exists a 
positive constant $f_0$ such that $f\geq f_0$ on $[0,R]$.
Let $u^\vepsi$ denote the unique monotone increasing classical solution to 
problem \eqref{moment1_radial}--\eqref{moment4_radial}.  
\begin{enumerate}
\item[(i)] If $n=2,3$, then either $u^\vepsi$ is strictly 
convex in $(0,R)$ or there exists an $\vepsi$-independent positive constant 
$c_0$ such that $u^\vepsi$ is strictly convex in $(0,R-c_0\vepsi)$.
\item [(ii)] If $n>3$, then there exists a monotone decreasing sequence 
$\{s_j\}_{j\geq 0} \subset (0,R)$ and two corresponding sequences 
$\{\vepsi_j\}_{j\geq 0} \subset (0,1)$, which is also monotone deceasing,
and $\{r_j^*\}_{j\geq 0} \subset (0,R)$ satisfying $s_j\searrow 0^+$ 
as $j\to \infty$ and $u^\vepsi_{rr}(s_j)\geq 0$ and $R-r_j^*=O(\vepsi)$ 
such that for each $j\geq 0$, $u^\vepsi$ is strictly 
convex in $(s_j, r_j^*)$ for all $\vepsi\in (0,\vepsi_j)$.
\end{enumerate}
\end{thm}

\begin{proof}
We divide the proof into three steps.

\smallskip
{\em Step 1}:
Let $w^\vepsi:=r^{n-1} u^\vepsi_r$ and 
$v^\vepsi:=\Del u^\vepsi=u^\vepsi_{rr} +\frac{n-1}{r} u^\vepsi_r=w^\vepsi_r$ be 
same as before, and define $\eta^\vepsi:=r^{n-1} u^\vepsi_{rr}$. On 
noting that 
$r^{n-1} v^\vepsi_r =\bigl(r^{n-1}u^\vepsi_{rr}\bigr)_r
-(n-1) r^{n-3} u^\vepsi_r
=\eta^\vepsi_r + \frac{1}{r} \eta^\vepsi -r^{n-2} v^\vepsi,$
the identity \eqref{v_eqn_0} can be rewritten as
\begin{align}\label{v_eqn_16}
-\vepsi \eta^\vepsi_{rr} 
+\Bigl[ \frac{2\vepsi}{r^2} +\frac{(u^\vepsi_r)^{n-1}}{r^{n-1}} \Bigr]\eta^\vepsi 
=r^{n-1}f + \vepsi(3-n)r^{n-3} v^\vepsi \qquad\mbox{in } (0,R).
\end{align}
So $\eta^\vepsi$ satisfies a linear uniformly elliptic equation. 

Clearly, $\eta^\vepsi(0)=0$. We claim that there exists (at least one)
$r_1\in (0,R]$ such that $\eta^\vepsi(r_1)\geq 0$. If not, then 
$\eta^\vepsi <0$ in $(0,R]$, and  $u^\vepsi_{rr}$ satisfies the same inequality. 
This implies  that $u^\vepsi_r$ is monotone decreasing in $(0,R]$. 
Since $u^\vepsi_r(0)=0$,  we have $u^\vepsi_r<0$ in $(0,R]$. But this contradicts with the fact 
that $u^\vepsi_r\geq 0$ in $(0,R]$. Therefore, the claim must be true.

Due to the factor $(3-n)$ in the second term on the right-hand side
of \eqref{v_eqn_16}, the situations for the cases $n\leq 3$ and 
$n>3$ are different, and need to be handled slightly different.

\smallskip
{\em Step 2: The case $n=2,3$}. Since $v^\vepsi\geq 0$, we have
\begin{equation}\label{v_eqn_17}
-\vepsi \eta^\vepsi_{rr} 
+\Bigl[ \frac{2\vepsi}{r^2} +\frac{(u^\vepsi_r)^{n-1}}{r^{n-1}} \Bigr]\eta^\vepsi 
\geq 0 \qquad\mbox{in } (0,R).
\end{equation}
Therefore, $\eta^\vepsi$ is a supersolution to a linear uniformly
elliptic differential operator.  By the weak maximum
principle (cf. \cite[page 329]{Evans98}) we have
$\min_{[0,r_1]} \eta^\vepsi (r) 
\geq \min\bigl\{0, \eta^\vepsi(0), \eta^\vepsi(r_1)\bigr\}
=\min\bigl\{0,\eta^\vepsi(r_1)\bigr\}=0.$

Let $r_*=\max\{r_1\in (0,R]; \, \eta^\vepsi(r_1)\geq 0\}$. By the above 
argument and the definition of $r_*$ we have $\eta^\vepsi\geq 0$ 
in $[0,r_*]$, $\eta^\vepsi(r_*)=0$ if $r_*\neq R$, 
and $\eta^\vepsi< 0$ in $(r_*, R]$. If $r_*=R$, then $\eta^\vepsi\geq 0$ 
in $[0,R]$. An application of the strong maximum principle to 
conclude that $\eta^\vepsi> 0$ in $(0,R)$.  
Hence, $u^\vepsi_{rr}> 0$ in $(0,R)$. Thus, $u^\vepsi$ is strictly convex 
in $(0,R)$. So the first part of the theorem's assertion is proved.

On the other hand, if $r_*< R$, we only know that $u^\vepsi$ is strictly 
convex in $(0,r_*)$. We now prove that $R-r_*=O(\vepsi)$. 
%
By \eqref{v_eqn_16} and the above setup we have
\begin{equation*} 
-\vepsi \eta^\vepsi_{rr} 
\geq r^{n-1}f \geq f_0 r^{n-1} \qquad\mbox{in } (r_*,R).
\end{equation*}
Integrating the above inequality over $(r_*, r)$ for $r\leq R$
and noting that $\eta^\vepsi_r(r_*)\leq 0$ we get
$-\vepsi \eta^\vepsi_r \geq \frac{f_0}{n} (r^n-r_*^n)$ in $(r_*,R).$
Integrating again over $(r_*,R)$ and using the fact that $\eta^\vepsi(r_*)=0$ and
the  algebraic inequality
$\frac{1}{n+1} \frac{R^{n+1}-r_*^{n+1}}{R-r_*}- r_*^n \geq \frac{1}{n+1} R^n$,
we arrive at
\[
-\vepsi R^{n-1} u_{rr}(R)=-\vepsi \eta^\vepsi(R) 
\geq \frac{f_0 R^n(R-r_*)}{n(n+1)}.
\]
It follows from \eqref{v_eqn_4} that
\begin{align*}
R-r_* \leq \frac{\vepsi n(n+1) |u_{rr}(R)|}{R f_0} 
&\leq \frac{\vepsi n(n+1)}{R^2 f_0} \Bigl[ \vepsi R 
+(n-1)\bigl(nL_f(R)\bigr)^{\frac{1}{n}} \Bigr] =:c_0\vepsi. 
\end{align*}
Thus, 
$R-r_*=O(\vepsi),$
and $u^\vepsi$ is strictly convex in $(0,R-c_0\vepsi)$.

\smallskip
{\em Step 3: The case $n>3$:} First, By the argument used in {\em Step 1}, it is
easy to show that $\eta^\vepsi$ can not be strictly negative in the whole of 
any neighborhood of $r=0$. Thus, there exists a monotone decreasing sequence
$\{s_j\}_{j\geq 0} \subset (0,R)$ such that $s_j\searrow 0^+$ as $j\to \infty$
and $\eta^\vepsi(s_j)\geq 0$.

Second, we note that
\begin{align*}
\eps(3-n)r^{n-3}v^\eps 
&= \eps(3-n)r^{n-3}\left(\ue_{rr}+\frac{n-1}r \ue_r\right)
=\eps(3-n)\left(\frac{\eta^\eps}{r^2}+(n-1)r^{n-4}\ue_r\right)
\end{align*}
Using this identity in \eqref{v_eqn_16}, we have
\begin{align*}
-\vepsi \eta^\vepsi_{rr} 
+\Bigl[ \frac{(n-1)\vepsi}{r^2} +\frac{(u^\vepsi_r)^{n-1}}{r^{n-1}} \Bigr]\eta^\vepsi 
=r^{n-4} [r^3 f + \vepsi(3-n)(n-1)\ue_r] \quad\mbox{in } (0,R).
\end{align*}
By (ii) of Theorem \ref{fine_estimates} we know that $u^\vepsi_r$ is uniformly
bounded in $[0,R]$.  Then for each $s_j$ there exists an $\vepsi_j>0$ (without
loss of the generality, choose $\vepsi_j< \vepsi_{j-1}$) 
such that for $\vepsi\in (0, \vepsi_j)$, there holds
$[r^3 f + \vepsi(3-n)(n-1)\ue_r] \ge 0$ in  $(s_j,R).$
Hence, $\eta^\vepsi$ is a supersolution to a linear uniformly elliptic
operator on $(s_j,R)$ for $\vepsi<\vepsi_j$. 

Third, for each fixed $j\geq 1$, let 
$r_j^*=\max\{r\in (s_j,R]; \, \eta^\vepsi(r)\geq 0\}$. 
Trivially, by the construction, $r_j^*\geq s_{j-1} > s_j$. 
By the weak maximum principle (cf. \cite[page 329]{Evans98}) we have
\begin{equation*}
\min_{[s_j,r_j^*]} \eta^\vepsi (r) 
\geq \min\bigl\{0, \eta^\vepsi(s_j), \eta^\vepsi(r_j^*)\bigr\}\geq 0.
\end{equation*}

Finally, repeating the argument of {\em Step 2:}, we conclude that $\ue$ is either
strictly convex in $(s_j,R)$ or in $(s_j,r_j^*)$ with $R-r_j^*=O(\vepsi)$ 
for $\vepsi\in (0,\vepsi_j)$. The proof is now complete.
\end{proof}


\section{Convergence of vanishing moment approximations} \label{section-5}
The goal of this section is to show that the solution $u^\vepsi$ 
of problem \eqref{moment_radial} converges
to the convex solution $u$ of problem \eqref{MA_radial} 
We present two different proofs for the convergence.  The 
first proof is based on the variational formulations of both problems.
The second proof, which can be extended to more general non-radially 
symmetric case, is done in the viscosity solution setting
\cite{Crandall_Ishii_Lions92}. Both proofs mainly rely on two key ingredients. 
The first is the solution estimates obtained in Theorem \ref{fine_estimates}, 
and the second is the uniqueness of solutions to problem 
\eqref{MA_radial}.

\begin{thm}\label{convergence_thm1}
Suppose $f\in C^0((0,R))$ and there exists a
positive constant $f_0$ such that $f\geq f_0$ in $[0,R]$.
Let $u$ denote the convex (classical) solution to problem \eqref{MA_radial}
and $u^\vepsi$ be the 
monotone increasing classical solution to 
problem \eqref{moment_radial}.  Then

\begin{enumerate}
\item[{\rm (i)}] $u^0=\lim_{\vepsi\to 0^+} u^\vepsi$ exists pointwise
and $u^\vepsi$ converges to $u^0$ uniformly in every compact subset 
of $(0,R)$ as $\vepsi\to 0^+$. Moreover, $u^0$ is strictly convex in 
every compact subset. Hence, it is convex in $[0,R]$.

\item[{\rm (ii)}] $u^\vepsi_r$ converges to $u^0_r$ weakly $*$  
in $L^\infty((0,R))$ as $\vepsi\to 0^+$. 

\item[{\rm (iii)}] $u^0\equiv u$.
\end{enumerate}

\end{thm}

\begin{proof}
It follows from (ii) of Theorem \ref{fine_estimates} that
$\|u^\vepsi\|_{C^1([0,R])}$ is uniformly bounded in $\vepsi$, 
then $\{u^\vepsi\}_{\vepsi\geq 0}$ is uniformly equicontinuous. By
Arzela-Ascoli compactness theorem (cf. \cite[page 635]{Evans98}) 
we conclude that there exists a subsequence of $\{u^\vepsi\}_{\vepsi\geq 0}$ 
(still denoted by the same notation) and $u^0\in C^1([0,R])$ such that
\begin{alignat*}{2}
u^\vepsi &\longrightarrow u^0
&&\qquad\mbox{uniformly in every compact set $E\subset (0,R)$ as } \vepsi\to 0^+,\\
u^\vepsi_r &\longrightarrow u^0_r
&&\qquad\mbox{weakly $*$ in $L^\infty((0,R))$ as } \vepsi\to 0^+,
\end{alignat*}
and $u^\vepsi(R)=g(R)$ implies that
\[
u^0(R)=g(R).
\]

In addition, by Theorem \ref{convexity_thm} we have that
for every compact subset $E\subset (0,R)$ there exists $\vepsi_0>0$ such that
$E\subset (0, R-c_0\vepsi)$ and $u^\vepsi$ is strictly convex
in $(0, R-c_0\vepsi)$ for $\vepsi<\vepsi_0$. 
It follows from a well-known property of convex functions
(cf. \cite{HUT01}) that $u^0$ must be strictly convex in $E$ and
$u^0\in C^{1,1}_{\mbox{\tiny loc}}((0,R))$.

Testing equation \eqref{v_eqn_0} with an arbitrary function
$\chi\in C^2_0((0,R))$ yields
\begin{align}\label{v_eqn_20}
\vepsi\int_0^R r^{n-1} v^\vepsi_r(r) \chi_r(r) \,dr  
-\frac{1}{n} \int_0^R \bigl(u^\vepsi_r(r)\bigr)^n \chi_r(r)\,dr
=\int_0^R r^{n-1} f(r) \chi(r)\, dr,
\end{align}
where as before $v^\vepsi=\Del u^\vepsi=u^\vepsi_{rr} +\frac{n-1}{r} u^\vepsi_r$.
By Schwartz inequality and (vi) of Theorem \ref{fine_estimates} we have
\begin{align*}
\vepsi\int_0^R r^{n-1} v^\vepsi_r(r) \chi_r(r) \,dr 
&=-\vepsi \int_0^R r^{n-1} v^\vepsi(r) \Bigl[ \chi_{rr}(r) 
+ \frac{n-1}{r} \chi_r(r)\Bigr]\, dr \\
&\leq \vepsi\Bigl(\int_0^R r^{2(n-1)}|v^\vepsi(r)|^2\,dr\Bigr)^{\frac12}
\Bigl(\int_0^R |\Del \chi(r)|^2 \,dr\Bigr)^{\frac12} \\
&\leq \sqrt{\vepsi\, C_5} \Bigl(\int_0^R |\Del \chi(r)|^2 \,dr\Bigr)^{\frac12}
\to 0 \quad\mbox{as } \vepsi\to 0^+.
\end{align*}

Setting $\vepsi \to 0^+$ in \eqref{v_eqn_20} and using the Lebesgue 
Dominated Convergence Theorem yields
\begin{align}\label{v_eqn_21}
-\frac{1}{n} \int_0^R \bigl(u^0_r(r)\bigr)^n \chi_r(r)\,dr
=\int_0^R r^{n-1} f(r) \chi(r)\, dr \qquad\forall \chi\in C^1_0((0,R)).
\end{align}
It also follows from a standard test function argument that
$u^0_r(0)=0.$
This means that $u^0\in C^1([0,R])\cap C^{1,1}_{\mbox{\tiny loc}}((0,R))$ 
is a convex weak solution to problem \eqref{MA_radial}. 
By the uniqueness of convex solutions of problem \eqref{MA_radial}, 
there must hold $u^0\equiv u$.

Finally, since we have proved that every convergent subsequence
of $\{u^\vepsi\}_{\vepsi\geq 0}$ converges to the unique convex 
classical solution $u$ of  problem \eqref{MA_radial}, the whole sequence 
$\{u^\vepsi\}_{\vepsi\geq 0}$ must converge to $u$.
The proof is complete.
\end{proof}

Next, we state and prove a different version of Theorem \ref{convergence_thm1}.
The difference is that we now only assume problem \eqref{MA_radial}
has a unique 
strictly convex viscosity solution and so the proof must be
adapted to the viscosity solution framework. 

\begin{thm}\label{convergence_thm2}
Suppose $f\in C^0((0,R))$ and there exists a
positive constant $f_0$ such that $f\geq f_0$ on $[0,R]$.
Let $u$ denote the strictly convex viscosity solution to problem
\eqref{MA_radial}
and $u^\vepsi$ be the
monotone increasing classical solution to
problem \eqref{moment_radial}.
Then
the statements of Theorem {\rm \ref{convergence_thm1}} still hold.
\end{thm}

\begin{proof}
Except the step of proving the variational formulation \eqref{v_eqn_21}, 
all other parts of the proof of Theorem \ref{convergence_thm1} are
still valid.  So we only need to show that $u^0$ is 
a viscosity solution of problem \eqref{MA_radial},
which is verified below by the definition of viscosity solutions. 

Let $\phi\in C^2([0,R])$ be strictly convex.
Suppose that $u^0-\phi$ 
has a local maximum at a point $r_0\in (0,R)$, that is, 
there exists a (small) number $\delta_0>0$ such that
$(r_0-\delta_0, r_0+\delta_0)\subset\subset (0,R)$ and 
\[
u^0(r)-\phi(r) \leq u^0(r_0)-\phi(r_0)
\qquad\forall r\in (r_0-\delta_0, r_0+\delta_0).
\]

Since $u^\vepsi$ (which still denotes a subsequence) converges 
to $u^0$ uniformly in $[r_0-\delta_0, r_0+\delta_0]$,
then for sufficiently small $\vepsi>0$, there exists
$r_\vepsi\in (0,R)$ such that $r_\vepsi\to r_0$ 
as $\vepsi\to 0^+$ and $u^\vepsi-\phi$ has a local 
maximum at $r_\vepsi$ (see \cite[Chapter 10]{Evans98}
for a proof of the claim). By elementary calculus, we have
$u^\vepsi_r(r_\vepsi)=\phi_r(r_\vepsi)$ and
$u^\vepsi_{rr}(r_\vepsi)\leq \phi_{rr}(r_\vepsi).$
Because both $u^\vepsi$ and $\phi$ are strictly convex,
there exists a (small) constant $\rho_0>0$ such that
for sufficiently small $\vepsi>0$ 
\begin{align*}
u^\vepsi_{rr}(r)\leq \phi_{rr}(r) \qquad\forall r\in (r_0-\rho,r_0+\rho),\,\,
\rho<\rho_0,
\end{align*}
which together with an application of Taylor's formula implies that
\begin{align*}
u^\vepsi_r(r)=\phi_r(r)+O(|r-r_\vepsi|)\qquad\forall r\in (r_0-\rho,r_0+\rho),\,\,
\rho<\rho_0.
\end{align*}

Let $\chi\in C^2_0((r_0-\rho,r_0+\rho))$ with $\chi \geq 0$ and
$\chi(r_0)>0$.  Testing \eqref{v_eqn_0} with $\chi$ yields
\begin{align}\label{v_eqn_22}
&\frac{1}{2n\rho}\int_{r_0-\rho}^{r_0+\rho} 
\bigl((\phi_r(r))^n\bigr)_r \chi(r) \,dr  
=\frac{1}{2\rho}\int_{r_0-\rho}^{r_0+\rho} 
(\phi_r(r))^{n-1} \phi_{rr}(r) \chi(r) \,dr \\
&\qquad \geq \frac{1}{2\rho }\int_{r_0-\rho}^{r_0+\rho} 
\bigl[(u^\vepsi_r(r))^{n-1} + O(|r-r_\vepsi|^{n-1})\bigr] 
u^\vepsi_{rr}(r) \chi(r)\,dr \nonumber \\
&\qquad =\frac{1}{2n\rho}\int_{r_0-\rho}^{r_0+\rho} 
\Bigl[ \bigl((u^\vepsi_r(r))^n\bigr)_r +O(|r-r_\vepsi|^{n-1})u^\vepsi_{rr}(r)
\Bigr] \chi(r) \,dr \nonumber \\
&\qquad \geq \frac{1}{2\rho }\int_{r_0-\rho}^{r_0+\rho} 
r^{n-1} \bigl[f(r) \chi(r)+\vepsi v^\vepsi_r(r) \chi_r(r) \bigr]\,dr \nonumber 
-C_9\rho^{n-2}\int_{r_0-\rho}^{r_0+\rho} u^\vepsi_r(r)
\chi_r(r)\,dr \nonumber 
\end{align}
for some positive $\rho$-independent constant $C_9$. Here, we have
used the fact that $u^\vepsi_{rr}\geq 0, \chi\geq 0$ in 
$[r_0-\rho, r_0+\rho]$ to get the last inequality.

From (vi) of Theorem \ref{fine_estimates}, we have
\begin{align}\label{v_eqn_23}
\vepsi  \int_{r_0-\rho}^{r_0+\rho} r^{n-1} v^\vepsi_r(r) \chi_r(r) \,dr 
& = -\vepsi  \int_{r_0-\rho}^{r_0+\rho} r^{n-1} v^\vepsi(r)\Bigl[
\chi_{rr}(r) +\frac{n-1}{r} \chi_r(r) \Bigr]\,dr  \\
& \leq \vepsi \Bigl(\int_0^R r^{2(n-1)} |v^\vepsi(r)|^2\, dr\Bigr)^{\frac12}
\Bigl(\int_0^R |\Del \chi(r)|^2\, dr\Bigr)^{\frac12} \nonumber\\
& 
\leq \sqrt{\vepsi\, C_5} \Bigl(\int_0^R |\Del \chi(r)|^2\, dr\Bigr)^{\frac12}. 
\nonumber
\end{align}

Setting $\vepsi\to 0^+$ in \eqref{v_eqn_22} and using 
\eqref{v_eqn_23} we get
\begin{align}\label{v_eqn_24}
&\frac{1}{2\rho}\int_{r_0-\rho}^{r_0+\rho} 
(\phi_r(r))^{n-1} \phi_{rr}(r) \chi(r) \,dr  \\
&\qquad \geq \frac{1}{2\rho }\int_{r_0-\rho}^{r_0+\rho} r^{n-1}f(r)\chi(r)\,dr
-C_9\rho^{n-2}\int_{r_0-\rho}^{r_0+\rho} u^0_r(r)
\chi_r(r)\,dr.  \nonumber
\end{align}
Here, we have used the fact that $u^\vepsi_r$ converges to $u^0_r$ 
weakly $*$ in $L^\infty((0,R))$ to pass to the limit in the
last term on the right-hand side.

Finally, letting $\rho\to 0^+$ in \eqref{v_eqn_24} and 
using the Lebesgue-Besicovitch Differentiation Theorem
(cf. \cite{Evans98}) we have
$
(\phi_r(r_0))^{n-1} \phi_{rr}(r_0) \chi(r_0) 
\geq  r_0^{n-1} f(r_0) \chi(r_0).$
Hence,
\[
\Bigl[\frac{\phi_r(r_0)}{r_0}\Bigr]^{n-1} \phi_{rr}(r_0) \geq f(r_0),
\]
so $u^0$ is a viscosity subsolution to equation \eqref{MA1_radial}.

Similarly, we can show that if $u^0-\phi$ assumes a local minimum
at $r_0\in (0,R)$ for a strictly convex function $\phi\in C^2_0((0,R))$,
there holds
\[
\Bigl[\frac{\phi_r(r_0)}{r_0}\Bigr]^{n-1} \phi_{rr}(r_0) \leq f(r_0).
\]
Therefore, $u^0$ is also a viscosity supersolution to equation \eqref{MA1_radial}.
Thus, it is a viscosity solution. The proof is complete.
\end{proof}

\section{Rates of convergence} \label{section-6}
In this section,  we derive rates of convergence for 
$u^\vepsi$ in various norms. Here we consider two cases,
namely, the $n$-dimensional radially symmetric case and 
the general $n$-dimensional case, under different assumptions. 
In both cases, the linearization of the Monge-Amp\`ere operator 
are explicitly exploited, and it plays a key role in our proofs. 

\begin{thm}\label{convergence_rate_thm1}
Let $u$ denote the strictly convex classical solution to problem
\eqref{MA_radial}
and $u^\vepsi$ be the
monotone increasing classical solution to
problem 
\eqref{moment_radial}.
Then
there holds the following estimates:
\begin{align}\label{rate_1}
\Bigl(\int_0^R \theta^\vepsi(r) |u_r(r)-u^\vepsi_r(r)|^2\, dr\Bigr)^{\frac12} 
&\leq \vepsi^{\frac34}\, C_{10},\\
\Bigl(\int_0^R r^{n-1}|\Del u(r)-\Del u^\vepsi(r)|^2\, dr\Bigr)^{\frac12} 
&\leq \vepsi^{\frac14}\,C_{11}, \label{rate_2} 
\end{align}
where $C_j=C_j(\|r^{n-1} \Del u_r\|_{L^2})$ for $j=10,11$ are two 
positive $\vepsi$-independent constants, and 
\begin{equation}\label{rate_6}
\theta^\vepsi(r): =\frac{(u_r)^n -(u^\vepsi_r)^n}{u-u^\vepsi}
=\sum_{j=0}^{n-1} (u_r(r))^j (u^\vepsi_r(r))^{n-1-j} > 0
\quad\mbox{\rm in } (0,R].
\end{equation}

\end{thm}

\begin{proof}
Let
$v:=\Del u=u_{rr}-\frac{n-1}{r} u_r$,  
$v^\vepsi:=\Del u^\vepsi=u^\vepsi_{rr}-\frac{n-1}{r} u^\vepsi_r,$
and $e^\vepsi:=u-u^\vepsi.$ 
On noting that \eqref{moment1_radial} can be written into \eqref{v_eqn_0},
multiplying \eqref{MA1_radial} by $r^{n-1}$ and subtracting the resulting 
equation from \eqref{v_eqn_0} yields the following error equation:
\begin{align}\label{rate_4}
\vepsi \bigl(r^{n-1} v^\vepsi_r\bigr)_r 
+\frac{1}{n} \bigl[ (u_r)^n -(u^\vepsi_r)^n \bigr]_r =0
\qquad\mbox{in } (0,R).
\end{align}
Testing \eqref{rate_4} with $e^\vepsi$ using boundary condition
$e^\vepsi_r(0)=e^\vepsi(R)=0$ we get 
\begin{align}\label{rate_5}
\vepsi \int_0^R r^{n-1} v^\vepsi_r(r) e^\vepsi_r(r)\,dr 
+\frac{1}{n} \int_0^R \theta^\vepsi(r) |e^\vepsi_r(r)|^2\,dr=0,
\end{align}
where $\theta^\vepsi$ is defined by \eqref{rate_6}.

Integrating by parts on the first term of \eqref{rate_5}
and rearranging terms we obtain
\begin{align}\label{rate_7}
\vepsi \int_0^R r^{n-1} |\Del e^\vepsi(r)|^2\, dr
&+\frac{1}{n} \int_0^R \theta^\vepsi(r) |e^\vepsi_r(r)|^2\, dr \\
&=\vepsi R^{n-1} \Del e^\vepsi(R) e^\vepsi_r(R) 
-\vepsi \int_0^R r^{n-1} v_r(r) e^\vepsi_r(r)\,dr.
\nonumber
\end{align}

We now bound the two terms on the right-hand side as follows. 
First, for the second term, a simple application of the Schwarz 
and Young's inequalities gives
\begin{align}\label{rate_8}
\vepsi \int_0^R r^{n-1} v_r(r) e^\vepsi_r(r)\,dr
&\leq \frac{1}{4n} \int_0^R \theta^\vepsi(r) |e^\vepsi_r(r)|^2\,dr 
+\vepsi^2 n \int_0^R \frac{r^{2(n-1)}}{\theta^\vepsi(r)} |v_r(r)|^2\,dr.
\nonumber
\end{align}
Second, to bound the first term on the right-hand side of
\eqref{rate_7}, we use the boundary condition $v^\vepsi(R)=\vepsi$ to get 
\[
|\Del e^\vepsi(R)| =|v(R)-v^\vepsi(R)|=|v(R)-\vepsi|\leq |v(R)|+1 =:M,
\]
and
\begin{align*}
\bigl|R^{n-1} e^\vepsi_r(R)\bigr|^2 
&= \int_0^R \bigl((r^{n-1} e^\vepsi_r(r))^2\bigr)_r\, dr
= 2\int_0^R  r^{2(n-1)} e^\vepsi_r(r)\, \Del e^\vepsi(r) \, dr \\
&\leq 2\Bigl(\int_0^R  r^{n-1} |\Del e^\vepsi(r)|^2\, dr\Bigr)^{\frac12}
 \Bigl(\int_0^R  r^{3(n-1)} |e^\vepsi_r(r)|^2\, dr\Bigr)^{\frac12}. 
\end{align*}
Hence by Young's inequality, we obtain
\begin{align}\label{rate_9}
|\vepsi R^{n-1} \Del e^\vepsi(R) e^\vepsi_r(R)| 
& \leq \sqrt{2}\vepsi M 
\Bigl(\int_0^R  r^{n-1} |\Del e^\vepsi(r)|^2\, dr\Bigr)^{\frac14}
\Bigl(\int_0^R  r^{3(n-1)} |e^\vepsi_r(r)|^2\, dr\Bigr)^{\frac14}  \\
& \leq \frac{\vepsi}{2} \int_0^R  r^{n-1} |\Del e^\vepsi(r)|^2\, dr
+ 2\vepsi M^{\frac43} \Bigl(\int_0^R r^{3(n-1)}|e^\vepsi_r(r)|^2\,dr\Bigr)^{\frac13} \nonumber\\
&\leq \frac{\vepsi}{2} \int_0^R  r^{n-1} |\Del e^\vepsi(r)|^2\, dr
 + \frac{1}{4n} \int_0^R \theta^\vepsi(r) |e^\vepsi_r(r)|^2\,dr 
 + \vepsi^{\frac32} n M^2 C \nonumber
\end{align}
for some $\vepsi$-independent constant $C=C(f,R,n)>0$.

Combining \eqref{rate_7}--\eqref{rate_9} yields
\begin{align}\label{rate_10}
\vepsi \int_0^R r^{n-1} |\Del e^\vepsi(r)|^2\, dr
&+\frac{1}{n} \int_0^R \theta^\vepsi(r) |e^\vepsi_r(r)|^2\, dr \\
&\leq 2\vepsi^2 n \int_0^R \frac{r^{2(n-1)}}{\theta^\vepsi(r)} |v_r(r)|^2\,dr
+\vepsi^{\frac32} n M^2 C. \nonumber
\end{align}
Thus, \eqref{rate_1} and \eqref{rate_2} follow from the fact that 
$\|r^{n-1} (\theta^\vepsi)^{-1}\|_{L^\infty}<\infty$.
\end{proof}

\begin{cor}\label{convergence_rate_thm1a}
Inequality \eqref{rate_1} implies that there exists an $\vepsi$-independent
constant $C>0$ such that
\begin{align}\label{rate_1a}
\Bigl(\int_0^R r^{n-1}|u_r(r)-u^\vepsi_r(r)|^2\, dr\Bigr)^{\frac12} 
\leq \vepsi^{\frac34}\, C C_{10}.
\end{align}
\end{cor}

Since the proof is simple, we omit it.

\begin{thm}\label{convergence_rate_thm2}
Under the assumptions of Theorem {\rm \ref{convergence_rate_thm1}},
there also holds the following estimate:
\begin{align}\label{rate_3}
\Bigl(\int_0^R r^{n-1}|u(r)-u^\vepsi(r)|^2\, dr\Bigr)^{\frac12} \leq \vepsi\, C_{12}
\end{align}
for some positive $\vepsi$-independent constant 
$C_{12}=C_{12}(R,n,u,C_{11})$.
\end{thm}

\begin{proof}
Let $\theta^\vepsi$ be defined by \eqref{rate_6}, and  $e^\vepsi$, $v$ and
$v^\vepsi$ be same as in Theorem \ref{convergence_rate_thm1}. Consider the 
following auxiliary problem:
\begin{subequations}
\label{MA_radial_lin}
\begin{align}\label{MA1_radial_lin}
\bigl( \theta^\vepsi \phi_r \bigr)_r 
&= n r^{n-1} e^\vepsi  \qquad\mbox{in } (0,R),\\
\phi(R) &=0, \ \ \psi_r(0)=0. \label{MA2_radial_lin} 
\end{align}
\end{subequations}
We note that the left-hand side of \eqref{MA1_radial_lin} is 
the linearization of \eqref{MA1_radial} at $\theta^\vepsi$.

Since $\theta^\vepsi>0$ in $(0,R]$, then \eqref{MA1_radial_lin} is a
linear elliptic equation. Using the fact that  
$c_1\geq r^{n-1}(\theta^\vepsi)^{-1} \geq c_0>0$ in $[0,R]$
for some $\vepsi$-independent positive constants $c_0$ and $c_1$,
it is easy to check that problem \eqref{MA_radial_lin}
has a unique classical solution $\phi$. Moreover,
\begin{align} \label{rate_3a}
\int_0^R r^{n-1} |\phi_{rr}(r)|^2\, dr 
+\int_0^R r^{n-1} |\phi_r(r)|^2\, dr
\leq \hat{C} \int_0^R r^{n-1} |e^\vepsi(r)|^2 \,dr 
\end{align}
for some $\vepsi$-independent constant $\hat{C}=\hat{C}(f,R,n,c_0,c_1)>0$.

Testing \eqref{MA1_radial_lin} by $e^\vepsi$, 
using the facts that $\phi_r(0)=\phi(R)=0$, $e^\vepsi(R)=0$ and 
$v^\vepsi(R)=\vepsi$ as well as error equation  \eqref{rate_4} we get
\begin{align}\label{rate_3b}
\int_0^R r^{n-1} |e^\vepsi(r)|^2 \, dr 
&=-\frac{1}{n}\int_0^R \theta^\vepsi(r) \phi_r(r) e^\vepsi_r(r)\, dr 
=\vepsi \int_0^R r^{n-1} v^\vepsi_r(r) \phi_r(r) \, dr\\ 
&\hspace{-0.25cm}=\vepsi  R^{n-1} v^\vepsi(R) \phi_r(R) 
  - \vepsi \int_0^R  r^{n-1} v^\vepsi(r) \Del \phi(r) \, dr \nonumber \\
&\hspace{-0.25cm}=\vepsi^2  R^{n-1} \phi_r(R) 
  +\vepsi \int_0^R  r^{n-1} [v(r)-v^\vepsi(r)] \Del \phi(r) \, dr \nonumber \\
&\hskip 1.0in 
  -\vepsi \int_0^R  r^{n-1} v(r) \Del \phi(r) \, dr, \nonumber
\end{align}
where we have used the short-hand notation $\Del\phi=r^{n-1}[ \phi_{rr} 
+(n-1) r^{-1} \phi_r ]$.

For each term on the right-hand side of \eqref{rate_3b} we have 
the following estimates:
\begin{align*}
&\vepsi^2 R^{n-1} \phi_r(R) 
= \frac{\vepsi^2}{R} \int_0^R (r^n \phi_r(r))_r \,dr \\
&\hskip 0.5in 
\leq \vepsi^2 R^{\frac{n}{2}} 
\Bigl(\int_0^R r^{n-1} |\phi_{rr}(r)|^2\, dr\Bigr)^{\frac12}
+\vepsi^2 \sqrt{n} R^{\frac{n-2}2}\Bigl(\int_0^R r^{n-1} |\phi_r(r)|^2\, dr\Bigr)^{\frac12},\\
&\vepsi \int_0^R  r^{n-1} [v(r)-v^\vepsi(r)] \Del \phi(r) \, dr \\
&\hskip 0.5in
\leq \vepsi \Bigl( \int_0^R r^{n-1} |v(r)-v^\vepsi(r)|^2\, dr\Bigr)^{\frac12}
\Bigl(\int_0^R r^{n-1} |\Del\phi(r)|^2\, dr\Bigr)^{\frac12}\\
&\hskip 0.5in
\leq \vepsi^{\frac54} C_{11} 
\Bigl(\int_0^R r^{n-1} |\Del\phi(r)|^2\, dr\Bigr)^{\frac12},\\
&-\vepsi \int_0^R  r^{n-1} v(r) \Del \phi(r) \, dr
\leq \vepsi\Bigl(\int_0^R r^{n-1} |v(r)|^2\, dr \Bigr)^{\frac12}
 \Bigl(\int_0^R r^{n-1} |\Del\phi(r)|^2\, dr \Bigr)^{\frac12}.  
\end{align*}
Substituting the above estimates into \eqref{rate_3b} and using
\eqref{rate_3a} we get
\begin{align}\label{rate_3c}
&\int_0^R r^{n-1} |e^\vepsi(r)|^2 \, dr \\
&\qquad
\leq \vepsi^2 R^{\frac{n}2} \left\{ 
\Bigl( \int_0^R r^{n-1} |\phi_{rr}(r)|^2\, dr\Bigr)^{\frac12} 
+\frac{\sqrt{n}}{R} \Bigl( \int_0^R r^{n-1} |\phi_r(r)|^2\, dr\Bigr)^{\frac12} \right\}
\nonumber \\
&\hskip 1in
+ \vepsi \bigl(\vepsi^{\frac14} C_{11}+C_u\bigr)
 \Bigl(\int_0^R r^{n-1} |\Del\phi(r)|^2\, dr \Bigr)^{\frac12} 
\nonumber \\
&\qquad \leq 4\vepsi \bigl(\vepsi R^{\frac{n}2} +\vepsi \sqrt{n} R^{-1} 
+\vepsi^{\frac14} C_{11} + C_u\bigr) \hat{C} 
\Bigl( \int_0^R r^{n-1} |e^\vepsi(r)|^2 \, dr \Bigr)^{\frac12} \nonumber
\end{align}
for some $\vepsi$-independent constant $C_u=C(u)>0$.

Hence, by \eqref{rate_3c} we conclude that \eqref{rate_3} holds 
with $C_{12}= 4\bigl(\vepsi R^{\frac{n}2} + \vepsi \sqrt{n} R^{-1}
+\vepsi^{\frac14} C_{11} + C_u\bigr) \hat{C}$.  The proof is complete.
\end{proof}

Since the proofs of Theorem \ref{convergence_rate_thm1} and
\ref{convergence_rate_thm2} only rely on the ellipticity of 
the linearization of the Monge-Amp\`ere operator,  hence,
the results of both theorems can be easily extended to
the general Monge-Amp\`ere problem \eqref{MAeqn} 
and its vanishing moment approximation \eqref{moment}. 

\begin{thm}\label{convergence_rate_thm3}
Let $u$ denote the strictly convex viscosity solution to problem
\eqref{MAeqn}
and $u^\vepsi$ be a 
classical solution to problem 
\eqref{moment}
Assume $u\in W^{2,\infty}(\Ome)\cap H^3(\Ome)$ and $u^\vepsi$ is either convex
or ``almost convex" in $\Ome$ (which means that $u^\vepsi$ is
convex in $\Ome$ minus an $\vepsi$-neighborhood of $\p\Ome$;
see Theorem {\rm \ref{convexity_thm}} for a precise description).  
Then there holds the following estimates:
\begin{align}\label{rate_14}
\Bigl(\int_\Ome |\nab u- \nab u^\vepsi|^2\, dx\Bigr)^{\frac12} 
&\leq \vepsi^{\frac34}\, C_{13},\\
\Bigl(\int_\Ome |\Del u-\Del u^\vepsi|^2\, dx\Bigr)^{\frac12} 
&\leq \vepsi^{\frac14}\,C_{14}, \label{rate_15}\\
\Bigl(\int_\Ome |u- u^\vepsi|^2\, dx\Bigr)^{\frac12} 
&\leq \vepsi\, C_{15},  \label{rate_16}
\end{align}
where $C_j=C_j(\|\nab \Del u\|_{L^2})$ for $j=13,14, 15$ are 
positive $\vepsi$-independent constants.

\end{thm}

\begin{proof}
Since the proof follows the exact same lines as those 
for Theorem \ref{convergence_rate_thm1}, we just briefly 
highlight the main steps.
First, the error equation \eqref{rate_4} is replaced by 
\begin{align}\label{rate_20}
\vepsi \Del v^\vepsi + \mbox{det}(D^2 u) -\mbox{det}(D^2 u^\vepsi) =0
\qquad\mbox{in } \Ome,
\end{align}
where $v^\vepsi=\Del u^\vepsi$.  
Next, equation \eqref{rate_6} becomes 
\begin{align}\label{rate_21}
\vepsi \int_\Ome |\Del e^\vepsi|^2\, dx
&+\int_\Ome \theta^\vepsi\nab e^\vepsi\cdot \nab e^\vepsi\, dx 
=\int_{\p \Ome} \Del e^\vepsi \frac{\p e^\vepsi}{\p \nu}\, dS
-\vepsi \int_\Ome \nab v \cdot\nab e^\vepsi\,dx,  
\end{align}
where
\begin{equation}\label{rate_22}
\theta^\vepsi= \Phi^\vepsi:=\mbox{\rm cof} (tD^2u+(1-t) D^2 u^\vepsi) 
\qquad\mbox{for some } t\in [0,1],
\end{equation}
now stands for the cofactor matrix of $tD^2u+(1-t) D^2 u^\vepsi$.
Since $u$ is assumed to be strictly convex, and $u^\vepsi$ is 
``almost convex",  there exists a positive constant $\theta_0$
such that 
%
$\theta^\vepsi\nab e^\vepsi\cdot \nab e^\vepsi\geq \theta_0 |\nab e^\vepsi|^2.$

It remains to derive a boundary estimate that is analogous to \eqref{rate_9}.
To the end, by the boundary condition $v^\vepsi|_{\p\Ome}=\vepsi$ and
the trace inequality, we have
\begin{align}\label{rate_23}
\int_{\p \Ome} \Del e^\vepsi \frac{\p e^\vepsi}{\p \nu}\, ds
&\leq \vepsi \bigl(\vepsi |\p\Ome| + \|\Del u\|_{L^2(\p\Ome)}^2 \bigr)^{\frac12} 
\Bigl\|\frac{\p e^\vepsi}{\p \nu}\Bigr\|_{L^2(\p\Ome)} \\
&\leq \vepsi M \|\nab e^\vepsi\|_{L^2(\Ome)}^{\frac12} 
\|\Del e^\vepsi\|_{L^2(\Ome)}^{\frac12} \nonumber \\
&\leq \frac{\vepsi}{2} \|\Del e^\vepsi\|_{L^2(\Ome)}^2 
+ M^{\frac43} \vepsi \|\nab e^\vepsi\|_{L^2(\Ome)}^{\frac23} \nonumber  \\
&\leq \frac{\vepsi}{2} \|\Del e^\vepsi\|_{L^2(\Ome)}^2 
+\frac{\theta_0}{4} \|\nab e^\vepsi\|_{L^2(\Ome)}^2 
+\frac{\vepsi^{\frac32} M^2}{\theta_0}. \nonumber
\end{align}
The desired estimates \eqref{rate_14} and \eqref{rate_15}
follow from combining \eqref{rate_21} and \eqref{rate_23}.

Finally, \eqref{rate_16} can be derived by using the same duality argument 
as that used in the proof of Theorem \ref{convergence_rate_thm2}.
We leave the details to the interested reader.
\end{proof}


\section{Numerical Experiments}\label{section-7}
In this section, we numerically verify the theoretical results given 
in the previous sections.  To this end, we solve 
\eqref{moment1_radial}--\eqref{moment4_radial}  in the domain $\Ome=(0,1)$.
We use the Hermite cubic finite element to construct our finite 
element space (cf. \cite{Brenner_Scott08}), and we use the following data:
\begin{align*}
f = (1+r^2)e^{n r^2/2},\qquad g(1) = e^{\frac12}.
\end{align*}
It can be readily checked that the exact solution is $u = e^{r^2/2}$.

We plot the computed solution and corresponding error 
in Figure \ref{Test614Figure1} with parameters
$n=4,\ \eps = 10^{-1}, h=4.0\times 10^{-3}$. We also plot the computed 
Laplacian, $\Del \ue:=\ue_{rr}+\frac{2}{r}\ue_r$, as well.
As shown by the pictures, the vanishing moment methodology accurately 
captures the convex solution in higher dimensions. Also, as expected, 
the Laplacian of $\ue$ is strictly positive 
(cf. Theorem \ref{convexity_thm1}).

Next, we plot both $u_r$ and $u_{rr}$ in two and four dimensions 
in Figures \ref{Test614Figure2}--\ref{Test614Figure3}
with $\eps$-values, $10^{-1}, 10^{-3}, 10^{-5}$.  Recall that the Hessian 
matrix of $\ue$ only has two distinct eigenvalues
$\ue_{rr}$ and $\frac1r\ue_{r}$.
As seen in Figure \ref{Test614Figure2}, $\ue_r$ is positive 
for all $\eps$-values and for both dimensions $n=2$ and $n=4$.
This result is in accordance with Corollary \ref{existence_cor}.  
Finally, Figure \ref{Test614Figure3} shows that
$\ue_{rr}$ is strictly positive except for a small $\eps$-neighborhood 
of the boundary, which agrees with
the theoretical results established in Theorem \ref{convexity_thm}.

\begin{figure}[htb]
\centerline{
\includegraphics[scale=0.11]{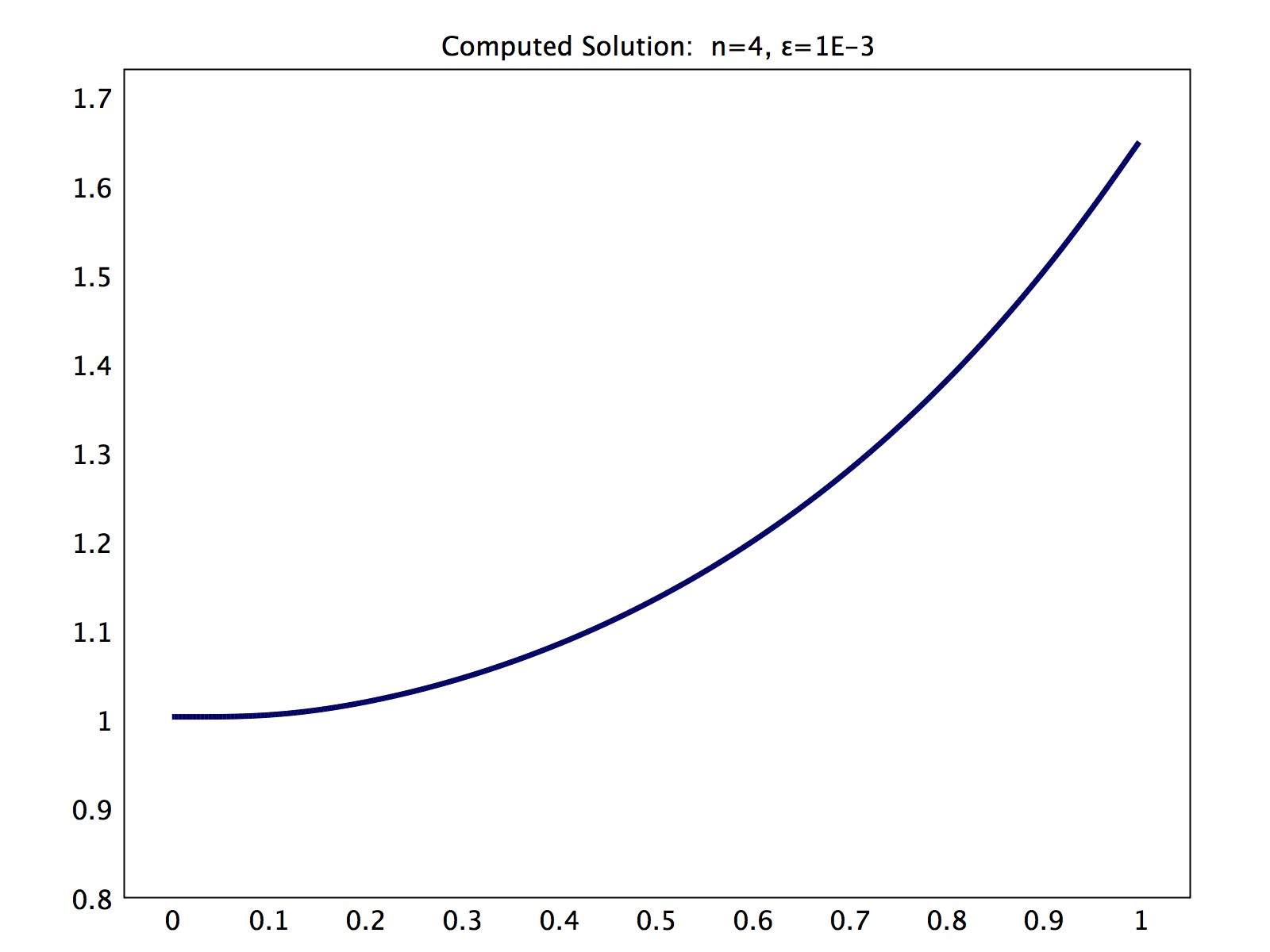} 
\includegraphics[scale=0.11]{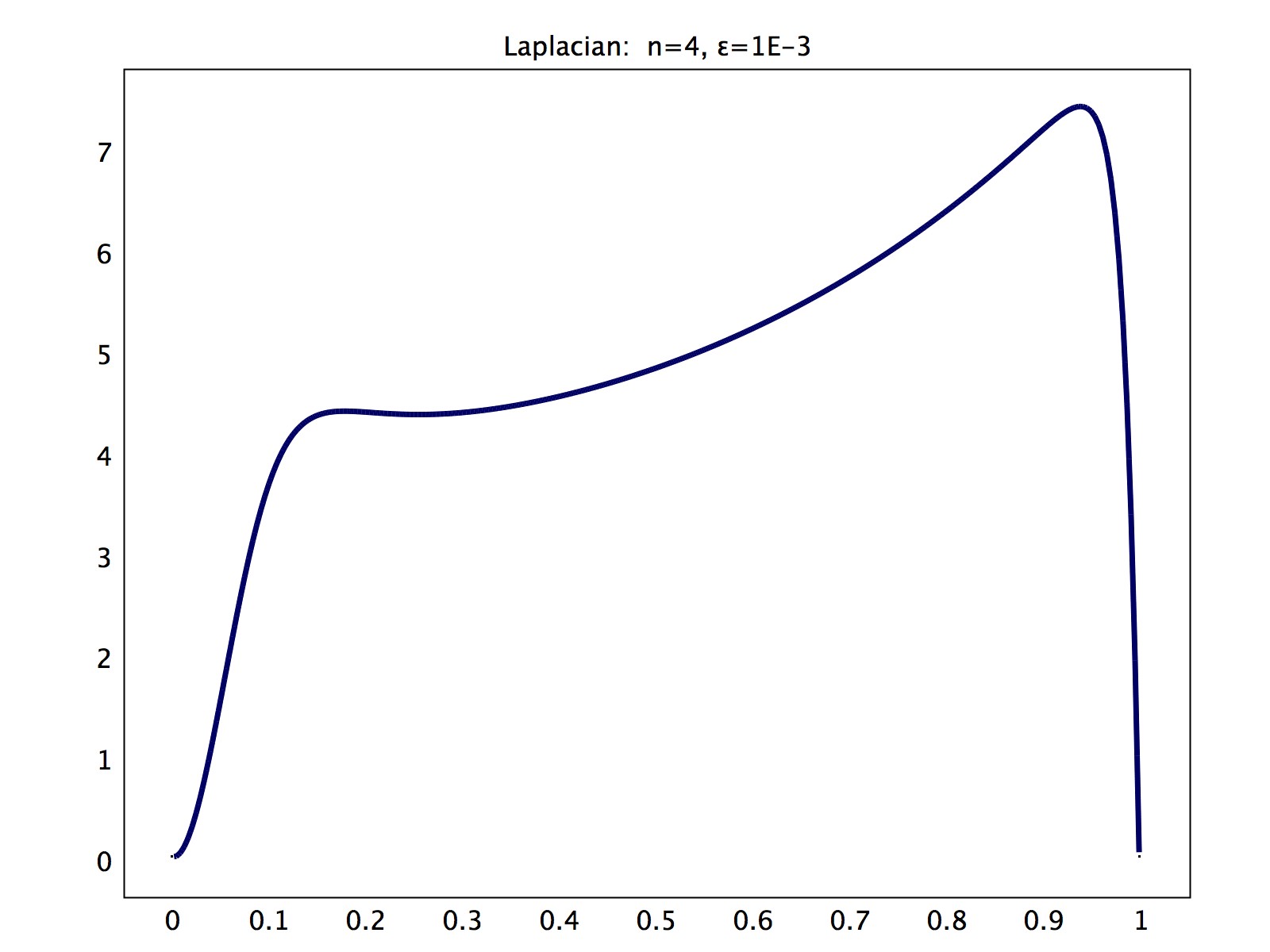} 
}
\caption{Computed solution of \eqref{moment_radial} 
 (left) and and computed 
Laplacian (right) with $n=4, \eps = 10^{-1}, h=4\times 10^{-3}$.}
\label{Test614Figure1}
\end{figure}

\begin{figure}[htb]
\centerline{
\includegraphics[scale=0.11]{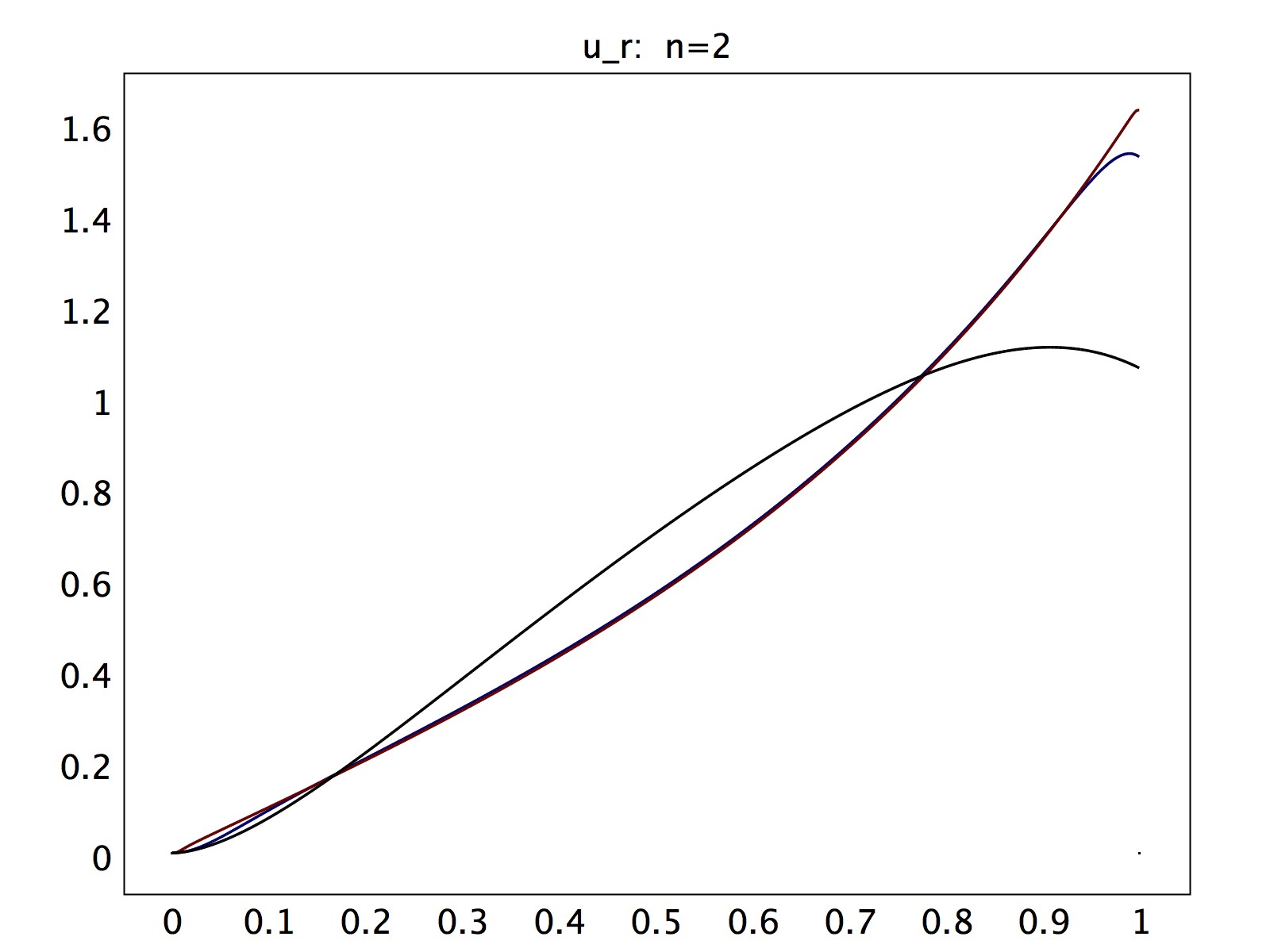}
\includegraphics[scale=0.11]{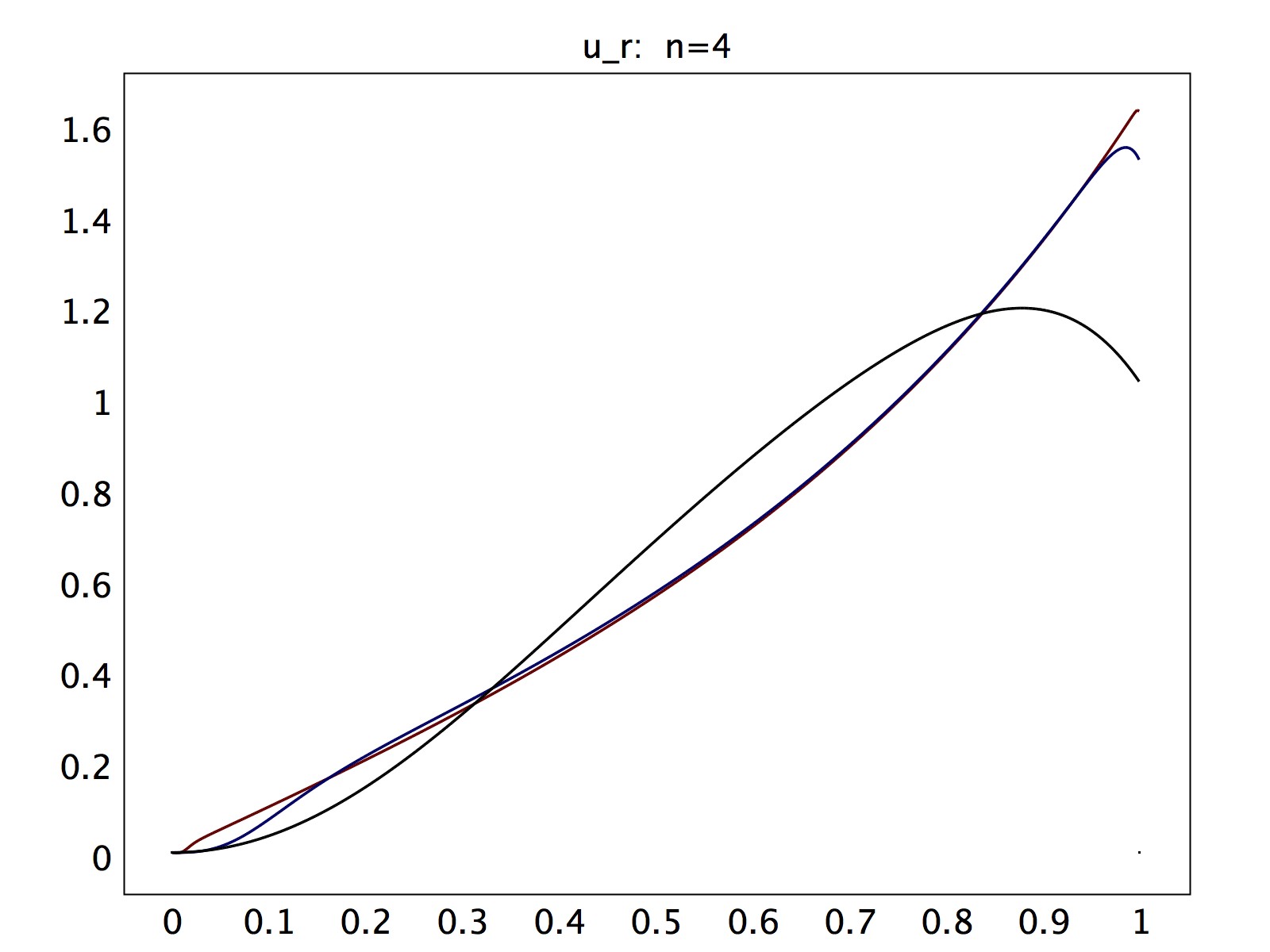}
}
\caption{Computed $u_{r}$ of \eqref{moment_radial} 
for $n=2$ (left), and $n=4$ 
(right) with $\eps=10^{-1}$ (black), $\eps=10^{-3}$ (blue), and 
$\eps = 10^{-5}$ (red) ($h=4\times 10^{-3}$).}
\label{Test614Figure2}
\end{figure}

\begin{figure}[htbp]
\centerline{
\includegraphics[scale=0.11]{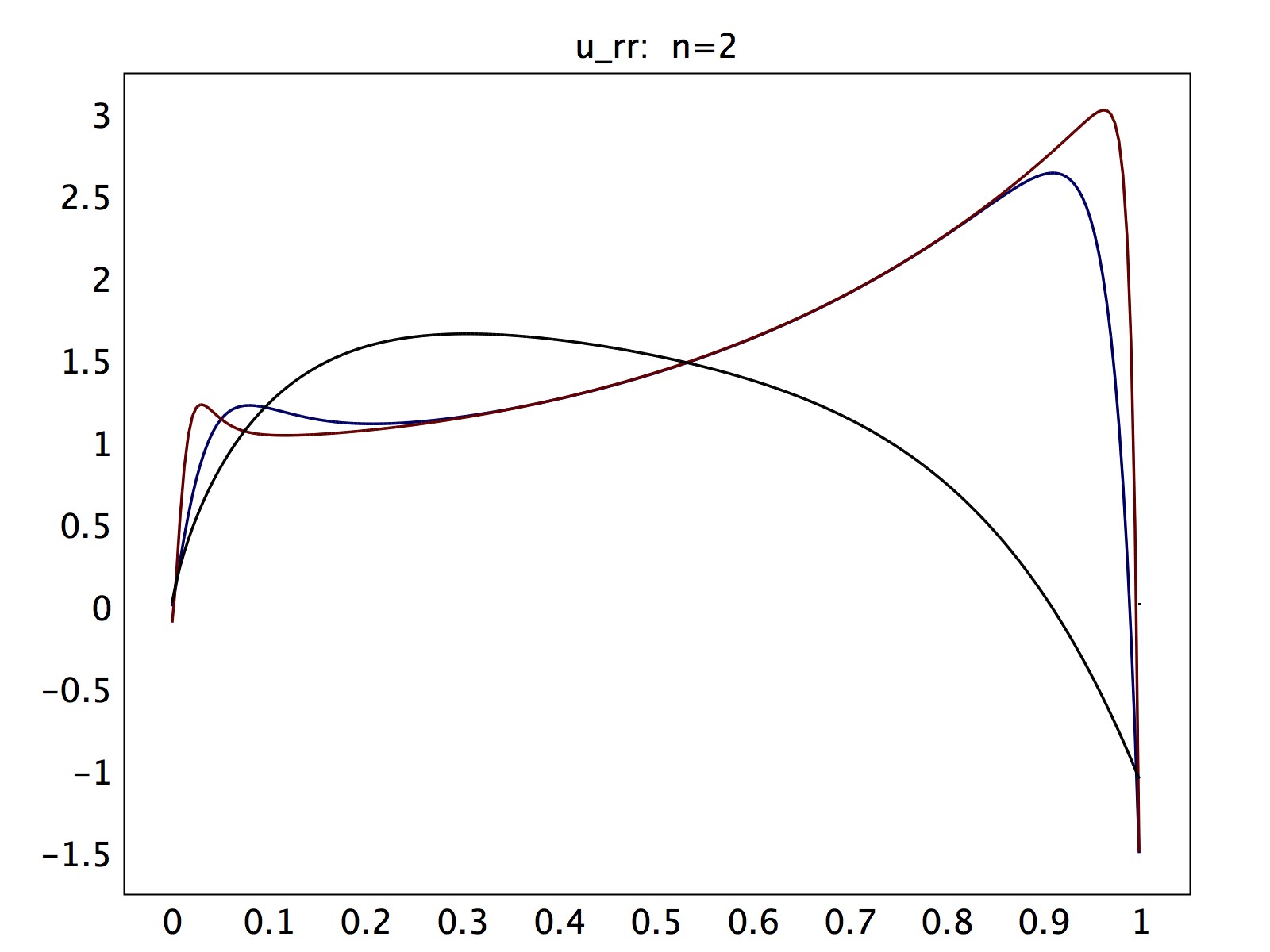}
\includegraphics[scale=0.11]{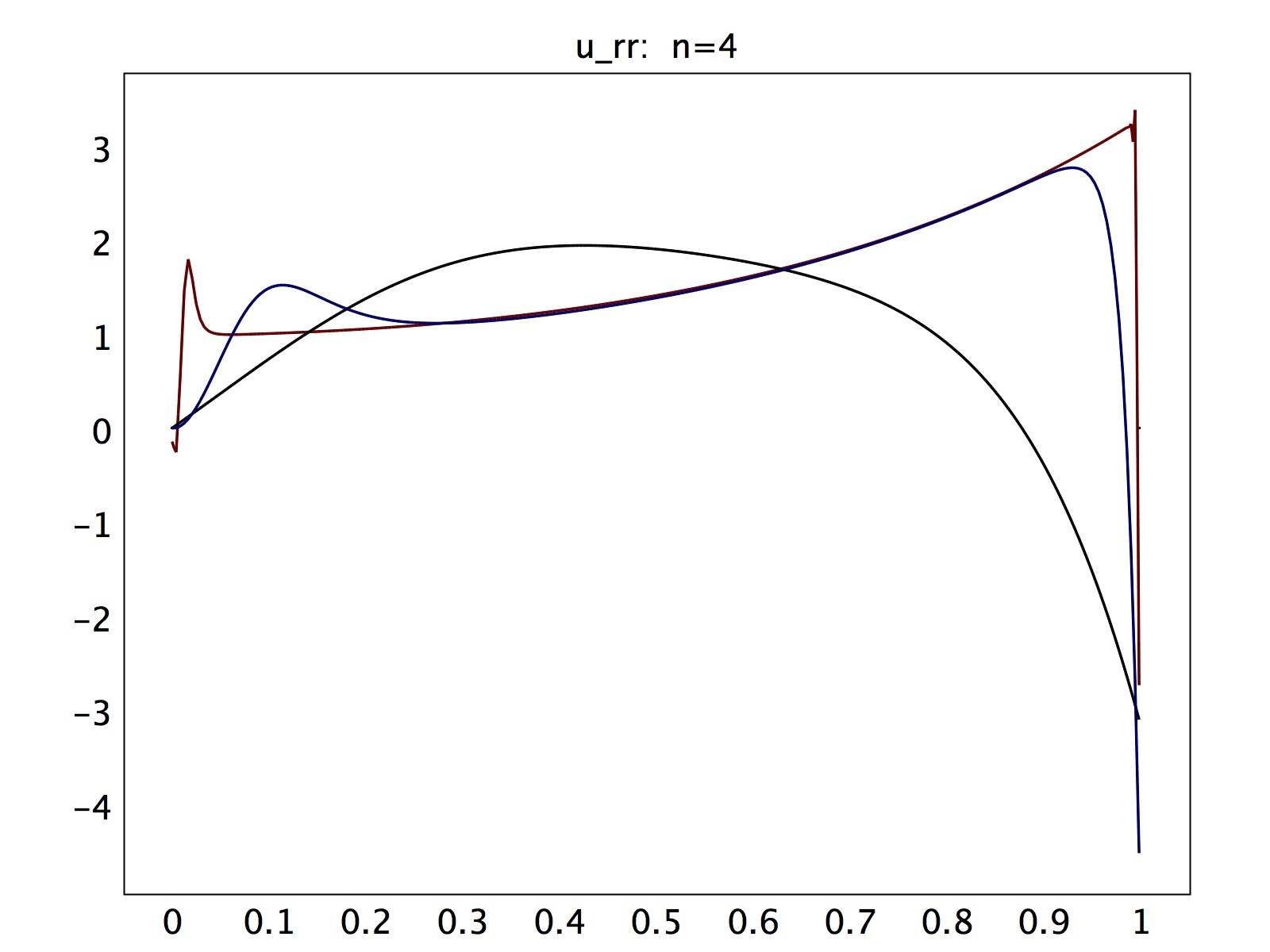}
}
\caption{Computed $u_{rr}$ of \eqref{moment_radial} 
for $n=2$ (left), and $n=4$ 
(right) with $\eps=10^{-1}$ (black), $\eps=10^{-3}$ (blue), 
and $\eps = 10^{-5}$ (red) ($h=4\times 10^{-3}$).}
\label{Test614Figure3}
\end{figure}

\section{Concluding Remarks} \label{section-8}

We like to comment on an interesting property of the vanishing moment 
method, namely, the ability of the vanishing moment method to approximate
the {\em concave} solution of the Monge-Amp\`ere problem \eqref{MAeqn}.
This can be achieved simply 
by letting $\vepsi\nearrow 0^-$ in \eqref{moment}.
This property can be easily proved as follows in the radially symmetric case.

Before giving the proof, we note that for a given $f> 0$ in $\Ome$,
equation \eqref{MAeqn1} does not a have concave solution in odd dimensions
(i.e., $n$ is odd) because $\mbox{det}(D^2 u)=f$ does not hold for any concave
function $u$ as all $n$ eigenvalues of Hessian $D^2u$ of a concave 
function $u$ must be nonpositive. On the other hand, in even dimensions (i.e., $n$ is
even), it is trivial to check that if $u$ is a convex solution of 
problem \eqref{MAeqn} 
with $g=0$, then $-u$, which is 
a concave function, must also be a solution of problem \eqref{MAeqn}.

Next, by the same token, it is easy to prove that if $u^\vepsi$ is a 
convex or ``almost convex" solution to problem \eqref{moment}, 
then $-u^\vepsi$, which is concave or 
``almost concave" in the sense that $-u^\vepsi$ is concave in $\Ome$ minus 
an $O(\vepsi)$-neighborhood of the boundary $\p\Ome$ of $\Ome$), 
must also be a solution of \eqref{moment}. 

Finally, let $n$ be a positive even integer,
it is easy to see that changing $u^\vepsi$ to $-u^\vepsi$ in 
\eqref{moment} 
is equivalent to changing 
$\vepsi$ to $-\vepsi$ in \eqref{moment}. 
For $\vepsi<0$, let $\delta:=-\vepsi$.  After replacing $\vepsi$ by $-\delta$
and $u^\vepsi$ by $\hat{u}^\delta:=-u^\vepsi$
in \eqref{moment_radial},
we see that 
$\hat{u}^\delta$ satisfies the same set of equations 
\eqref{moment_radial}
with $\delta (>0)$
in place of $\vepsi$. Hence, by the analysis of 
Sections \ref{section-2}--\ref{section-6} we know that
there exists a monotone increasing solution $\hat{u}^\delta$ 
to problem \eqref{moment_radial}
with
$\vepsi$ being replaced by $\delta$, which satisfies all the properties 
proved in Sections \ref{section-2}--\ref{section-6}.
Translating all these to $u^\vepsi=-\hat{u}^\delta$, we conclude that 
problem \eqref{moment_radial}
for $\vepsi<0$ has
a monotone decreasing solution which is either concave or ``almost concave"
in $(0,R)$ and converges to the unique concave solution of 
problem \eqref{MAeqn}
as $\vepsi\nearrow 0^-$. 
In addition, $u^\vepsi$ satisfies the error estimates stated 
in Theorems \ref{convergence_rate_thm1} and \ref{convergence_rate_thm2}.


\end{document}